\documentclass[12pt]{elsarticle}

\usepackage{amssymb}
\usepackage{amsmath}
\usepackage{amsthm}

\usepackage{hyperref}
\usepackage{cleveref}

\newcommand{\pref}[1]{(\ref{#1})}
\newcommand{\fullcref}[2]{\cref{#1}\pref{#1-#2}}
\newcommand{\csee}[1]{(see \cref{#1})}

\newcommand{\iso}{\cong}
\newcommand{\normal}{\triangleleft}

\renewcommand{\th}{^{\text{th}}}

\newcommand{\integer}{\mathbb{Z}}

\renewcommand{\pmod}[1]{\ (\mathop{\mathrm{mod}}{#1})}

\DeclareMathOperator{\Cay}{Cay}

\DeclareMathOperator{\Aut}{Aut}

\newcommand{\qt}[1]{\overline{#1}}
\newcommand{\bya}[1]{&\stackrel{\textstyle a}{\longrightarrow}&\qt{#1}}
\newcommand{\byai}[1]{&\stackrel{\textstyle a^{-1}}{\longrightarrow}&\qt{#1}}
\newcommand{\byb}[1]{&\stackrel{\textstyle b}{\longrightarrow}&\qt{#1}}
\newcommand{\bybi}[1]{&\stackrel{\textstyle b^{-1}}{\longrightarrow}&\qt{#1}}
\newcommand{\byc}[1]{&\stackrel{\textstyle c}{\longrightarrow}&\qt{#1}}
\newcommand{\byci}[1]{&\stackrel{\textstyle c^{-1}}{\longrightarrow}&\qt{#1}}

\crefformat{section}{Section~#2#1#3} 

\numberwithin{equation}{section}

\crefformat{claim}{Claim~#2#1#3}

 \newcounter{case}
 \newenvironment{case}[1][\unskip]{\refstepcounter{case}\bf
 \medskip \noindent Case \thecase\ #1. \it}{\unskip\upshape}
 \renewcommand{\thecase}{\arabic{case}}
\crefformat{case}{Case~#2#1#3}

 \newcounter{subcase}
 \newenvironment{subcase}[1][\unskip]{\refstepcounter{subcase}\bf
 \medskip \noindent \hskip\parindent Subcase \thesubcase\ #1. \it}{\unskip\upshape}
\crefformat{subcase}{Subcase~#2#1#3}
\renewcommand{\thesubcase}{\thecase.\arabic{subcase}}
\numberwithin{subcase}{case}

 \newcounter{subsubcase}
 \newenvironment{subsubcase}[1][\unskip]{\refstepcounter{subsubcase}\bf
 \medskip \noindent \hskip2\parindent Subsubcase \thesubsubcase\ #1. \it}{\unskip\upshape}
\crefformat{subsubcase}{Subsubcase~#2#1#3}
\renewcommand{\thesubsubcase}{\thesubcase.\arabic{subsubcase}}
\numberwithin{subsubcase}{subcase}

 \newcounter{subsubsubcase}
 \newenvironment{subsubsubcase}[1][\unskip]{\refstepcounter{subsubsubcase}\bf
 \medskip \noindent \hskip3\parindent Subsubsubcase \thesubsubsubcase\ #1. \it}{\unskip\upshape}
\crefformat{subsubsubcase}{Subsubsubcase~#2#1#3}
\renewcommand{\thesubsubsubcase}{\thesubsubcase.\arabic{subsubsubcase}}
\numberwithin{subsubsubcase}{subsubcase}

\theoremstyle{plain}
\newtheorem{thm}[equation]{Theorem}
\newtheorem{prop}[equation]{Proposition}
\newtheorem{cor}[equation]{Corollary}

\newtheorem{lem}[equation]{Lemma}
\newtheorem{FGL}[equation]{Lemma}

\crefformat{thm}{Theorem~#2#1#3}
\crefformat{prop}{Proposition~#2#1#3}
\crefformat{cor}{Corollary~#2#1#3}
\crefformat{goal}{Goal~#2#1#3}
\crefformat{lem}{Lemma~#2#1#3}
\crefformat{FGL}{the Factor Group Lemma~#2(#1)#3}

\theoremstyle{definition}
\newtheorem*{notation}{Notation}

\theoremstyle{remark}

\newtheorem*{ack}{Acknowledgments}

\crefformat{rem}{Remark~#2#1#3}

\newcommand{\olditemize}{} 
\let\olditemize\itemize
\def\itemize{%
	\olditemize
	\setlength{\itemsep}{0pt plus 0.01pt \relax}%
	}


\usepackage{color}
\newcommand{\refnote}[1]{\marginpar{%
	\color{blue}
	\vbox to 0pt{\vss
	$\begin{pmatrix} \text{note \cref{#1}} \end{pmatrix}$%
	\vskip -4pt}}}
\theoremstyle{definition}
\newtheorem{aid}[equation]{}
\newcommand{\oldendaid}{}
\let\oldendaid=\endaid
\renewcommand{\endaid}{\oldendaid\bigskip\hrule width\textwidth \bigbreak}

\journal{Discrete Mathematics}

\begin{document}

\begin{frontmatter}

\title{Cayley graphs of order $30p$ are hamiltonian}

\author{Ebrahim Ghaderpour}

\ead{Ebrahim.Ghaderpoor@uleth.ca}

\author{Dave Witte Morris}

\address{Department of Mathematics and Computer Science,
University of Lethbridge, Lethbridge, Alberta, T1K~3M4, Canada}

\ead{Dave.Morris@uleth.ca}
\ead[URL]{http://people.uleth.ca/~dave.morris/}

\begin{abstract}
Suppose $G$ is a finite group, such that $|G| = 30p$, where $p$~is prime. We show that if $S$ is any generating set of~$G$, then there is a hamiltonian cycle in the corresponding Cayley graph $\Cay(G;S)$.
\end{abstract}

\end{frontmatter}

\section{Introduction}

There is a folklore conjecture that every connected Cayley graph has a hamiltonian cycle. (See the surveys \cite{CurranGallian-survey,PakRadoicic-survey,WitteGallian-survey} for some background on this question.) The papers \cite{JungreisFriedman} and \cite{M2Slovenian-LowOrder} began a systematic study of this conjecture in the case of Cayley graphs for which the number of vertices has a prime factorization that is small and easy. In particular, combining several of the results in \cite{M2Slovenian-LowOrder} with \cite{CurranMorris2-16p,GhaderpourMorris-27p} and this paper shows:

\begin{center} \it
If\/ $|G| = kp$, where $p$~is prime, with $1 \le k < 32$ and $k \neq 24$, 
\\
then every connected Cayley graph on~$G$ has a hamiltonian cycle.
\end{center}

This paper's contribution to the project is the case $k = 30$:

\begin{thm} \label{30p}
If\/ $|G| = 30p $, where $p$~is prime, then every connected Cayley graph on~$G$ has a hamiltonian cycle.
\end{thm}

\begin{ack}
This work was partially supported by research grants from the Natural Sciences and Engineering Research Council of Canada.
\end{ack}

\section{Preliminaries}

Before proving \cref{30p}, we present some useful facts about hamiltonian cycles in Cayley graphs.

\begin{notation}
Throughout this paper, $G$ is a finite group.
	\begin{itemize}
	\item For any subset~$S$ of~$G$, $\Cay(G;S)$ denotes the \emph{Cayley graph} of~$G$ with respect to~$S$. Its vertices are the elements of~$G$, and there is an edge joining~$g$ to~$gs$ for every $g \in G$ and $s \in S$.

	\item For $x,y \in G$:
		\begin{itemize}
		\item $[x,y]$ denotes the \emph{commutator} $x^{-1} y^{-1} x y$,
		and
		\item $y^x$ denotes the \emph{conjugate} $x^{-1} y x$.
		\end{itemize}
	\item $\langle A \rangle$ denotes the subgroup generated by a subset~$A$ of~$G$.
	\item $G'$ denotes the \emph{commutator subgroup} $[G,G]$ of~$G$.
	\item $Z(G)$ denotes the \emph{center} of~$G$.
	\item $G \ltimes H$ denotes a \emph{semidirect product} of the groups $G$ and~$H$.
	\item $D_{2n}$ denotes the \emph{dihedral group} of order~$2n$.
	\item For $S \subset G$, a sequence $(s_1,s_2,\ldots,s_n)$ of elements of $S \cup S^{-1}$ specifies the walk in the Cayley graph $\Cay(G;S)$ that visits (in order) the vertices
		$$ e, s_1, s_1s_2, s_1s_2s_3,\ldots,s_1s_2\ldots s_n .$$
	If $N$ is a normal subgroup of~$G$, we use $(\qt{s_1},\qt{s_2},\ldots,\qt{s_n})$ to denote the image of this walk in the quotient $\Cay(G/N;S)$.
	\item If the walk $(\qt{s_1},\qt{s_2},\ldots,\qt{s_n})$ in $\Cay(G/N;S)$ is closed, then its \emph{voltage} is the product $s_1s_2\ldots s_n$. This is an element of~$N$.
	\item For $k \in \integer^+$, we use $(s_1,\ldots,s_m)^k$ to denote the concatenation of $k$~copies of the sequence $(s_1,\ldots,s_m)$. Abusing notation, we often write $s^k$ and $s^{-k}$ for 
		$$ \text{$(s)^k = (s,s,\ldots,s)$ \ and \ $(s^{-1})^k = (s^{-1},s^{-1},\ldots,s^{-1})$} , $$
respectively. Furthermore, we often write $\bigl( (s_1,\ldots,s_m), (t_1,\ldots,t_n) \bigr)$ to denote the concatenation $(s_1,\ldots,s_m, t_1,\ldots,t_n)$. For example, we have
		$$ \bigl( (a^2, b)^2, c^{-2} \bigr)^2
		= (a,a,b,a,a,b, c^{-1},c^{-1},a,a,b,a,a,b, c^{-1},c^{-1}) .$$
	\end{itemize}
\end{notation}

\begin{thm}[Maru\v si\v c, Durnberger, Keating-Witte \cite{KeatingWitte}] \label{KeatingWitte}
 If $G'$ is a cyclic group of prime-power order, then
every connected Cayley graph on~$G$ has a hamiltonian cycle.
 \end{thm}

\begin{FGL}[{``Factor Group Lemma''  {\cite[\S2.2]{WitteGallian-survey}}}] \label{FGL}
 Suppose
 \begin{itemize}
 \item $S$ is a generating set of~$G$,
 \item $N$ is a cyclic, normal subgroup of~$G$,
 \item $\qt{C} = (\qt{s_1}, \qt{s_2}, \ldots, \qt{s_n})$ is a hamiltonian cycle in $\Cay(G/N;S)$,
 and
 \item the voltage of~$\qt{C}$ generates~$N$.
 \end{itemize}
 Then $(s_1,\ldots,s_n)^{|N|}$ is a hamiltonian cycle in $\Cay(G;S)$.
 \end{FGL}

The following easy consequence of \cref{FGL} is well known (and is implicit in \cite{Marusic-HamCircCayGrfs}).

\begin{cor} \label{DoubleEdge}
Suppose
 \begin{itemize}
 \item $S$ is a generating set of~$G$,
 \item $N$ is a normal subgroup of~$G$, such that $|N|$ is prime,
 \item $s \equiv t \pmod{N}$ for some $s,t \in S \cup S^{-1}$ with $s \neq t$,
 and
 \item there is a hamiltonian cycle in $\Cay(G/N;S)$ that uses at least one edge labeled~$s$.
 \end{itemize}
 Then there is a hamiltonian cycle in $\Cay(G;S)$.
 \refnote{DoubleEdgeAid}
 \end{cor}

\begin{thm}[Alspach {\cite[Cor.~5.2]{Alspach-lifting}}]
 \label{AlspachSemiProd}
 If $G = \langle s \rangle \ltimes \langle t \rangle$, for some elements $s$ and~$t$ of~$G$, then $\Cay \bigl( G; \{s,t\} \bigr)$ has a hamiltonian cycle.
 \end{thm}

\begin{lem}[{}{\cite[Lem.~2.27]{M2Slovenian-LowOrder}}]
\label{NormalEasy}
 Let $S$ generate the finite group~$G$, and let $s \in S$, such that $\langle s \rangle \normal G$. If $\Cay \bigl( G/\langle s \rangle ; S \bigr)$ has a hamiltonian cycle,
 and
 either
 \begin{enumerate}
 \item \label{NormalEasy-Z} 
 $s \in Z(G)$,
 or
 \item \label{NormalEasy-notZ} 
 $Z(G) \cap \langle s \rangle = \{e\}$,
 \end{enumerate}
 then $\Cay(G;S)$ has a hamiltonian cycle.
 \end{lem}

 \begin{lem} \label{HamConnInSubgrp}
 Suppose 
 	\begin{itemize}
	\item $G = \langle a \rangle \ltimes \langle S_0\rangle$, where $\langle S_0\rangle$ is an abelian subgroup of odd order,
	\item $\#(S_0 \cup S_0^{-1}) \ge 3$,
	and
	\item $\langle S_0 \rangle$ has a nontrivial subgroup~$H$, such that $H \normal G$ and $H \cap Z(G) = \{e\}$.
	\end{itemize}
Then $\Cay \bigl( G; S_0 \cup \{a\} \bigr)$ has a hamiltonian cycle.
\end{lem}

\begin{proof}
Since $\langle S_0 \rangle$ is abelian of odd order, and  $\#(S_0 \cup S_0^{-1}) \ge 3$, we know that $\Cay \bigl( \langle S_0 \rangle ; S_0 \bigr)$ is hamiltonian connected \cite{ChenQuimpo-hamconn}. Therefore, it has a hamiltonian path $(s_1,s_2,\ldots,s_{m})$, such that $s_1 s_2\cdots s_{m} \in H$. Then
	$$\bigl( s_1,s_2,\ldots,s_{m} ,a \bigr)^{|a|} $$
is a hamiltonian cycle in $\Cay \bigl( G;S_0 \cup \{a\} \bigr)$.
\refnote{HamConnInSubgrpRef}
\end{proof}

\begin{lem}[{}{\cite[Cor.~4.4]{CurranMorris2-16p}}] \label{G'2gen} 
If $a,b \in G$, such that $G = \langle a , b \rangle$, then $G' = \langle [a,b] \rangle$.
\end{lem}

\begin{lem}[{}{\cite[Prop.~5.5]{Witte-HamCircCayDiag}}] \label{Dihedral2pqr}
If $p$, $q$, and~$r$ are prime, then every connected Cayley graph on the dihedral group $D_{2pqr}$ has a hamiltonian cycle.
\end{lem}

\begin{lem} \label{DihedralxZr}
If $G = D_{2pq} \times \integer_r$, where $p$, $q$, and~$r$ are distinct odd primes, then every connected Cayley graph on~$G$ has a hamiltonian cycle.
\end{lem}

\begin{proof}
Let $S$ be a minimal generating set of~$G$, let $\varphi \colon G \to D_{2pq}$ be the natural projection, and let $T$ be the group of rotations in $D_{2pq}$, so $T = \integer_p \times \integer_q$.

For $s \in S$, we may assume:
	\begin{itemize}
	\item If $\varphi(s)$ has order~$2$, then $s = \varphi(s)$ has order~$2$. (Otherwise, \cref{DoubleEdge} applies with $t = s^{-1}$.)
	\item $\varphi(s)$~is nontrivial. (Otherwise, $s \in \integer_r \subset Z(G)$, so \fullcref{NormalEasy}{Z} applies.)
	\end{itemize}

Since $\varphi(S)$ generates $D_{2pq}$, it must contain at least one reflection (which is an element of order~$2$). So $S \cap D_{2pq}$ contains a reflection.

\setcounter{case}{0}

\begin{case}
Assume $S \cap D_{2pq}$ contains only one reflection.
\end{case}
Let $a \in S \cap D_{2pq}$, such that $a$ is a reflection.

Let $S_0 = S \smallsetminus \{a\}$. Since $\langle S_0 \rangle$ is a subgroup of the cyclic, normal subgroup $T \times \integer_r$, we know $\langle S_0 \rangle$ is normal. Therefore $G = \langle a \rangle \ltimes \langle S_0 \rangle$, so:
	\begin{itemize}
	\item If $\#S_0 = 1$, then \cref{AlspachSemiProd} applies.
	\item If $\#S_0 \ge 2$, then \cref{HamConnInSubgrp} applies with $H = T$, because $T \times \integer_r$ is abelian of odd order.
	\end{itemize}

\begin{case}
Assume $S \cap D_{2pq}$ contains at least two reflections.
\end{case}
Since no minimal generating set of $D_{2pq}$ contains three reflections,\refnote{Not3ReflectionsRef}
 the minimality of~$S$ implies that $S \cap D_{2pq}$ contains exactly two reflections; say $a$ and $b$ are reflections.

Let $c \in S \smallsetminus D_{2pq}$, so $\integer_r \subset \langle c \rangle$. Since $|c| > 2$, we know $\varphi(c)$ is not a reflection, so $\varphi(c) \in T$. The minimality of $S$ (combined with the fact that $\#S > 2$) implies $\langle \varphi(c) \rangle \neq T$.\refnote{phi(c)neqTRef} 
Since $\varphi(c)$ is nontrivial, this implies we may assume $\langle \varphi(c) \rangle = \integer_p$ (by interchanging $p$ and~$q$ if necessary).  Hence, we may write
	$$ \text{$c = w z$ with $\langle w \rangle = \integer_p$ and $\langle z \rangle = \integer_r$} .$$

We now use the argument of \cite[Case~5.3, p.~96]{KeatingWitte}, which is based on ideas of D.\,Maru\v si\v c \cite{Marusic-HamCircCayGrfs}. Let 
	$$\qt{G} = G/\integer_p = \qt{D_{2pq}} \times \integer_r = \qt{D_{2pq}} \times \langle \qt{c} \rangle .$$
Then $\qt{D_{2pq}} \iso D_{2q}$, so $(a,b)^q$ is a hamiltonian cycle in $\Cay \bigl( \qt{D_{2pq}} ; a,b \bigr)$. With this in mind, it is easy to see that
	$$ \Bigl(c^{r-1}, a, \bigl( (b,a)^{q-1},  c^{-1}, (a,b)^{q-1}, c^{-1} \bigr)^{(r-1)/2}, (b,a)^{q-1}, b \Bigr) .$$
is a hamiltonian cycle in $\Cay(\qt{G};S)$.\refnote{D2qxZrRef} This contains the string 
	$$\bigl( c,a, (b,a)^{q-1},  c^{-1}, a \bigr) ,$$
which can be replaced with the string
	$$ \bigl( b,c, (b,a)^{q-1}, b, c^{-1} \bigr) $$
to obtain another hamiltonian cycle.\refnote{D2qxZrOtherRef}
Since 
	\begin{align*}
	 ca (ba)^{q-1} c^{-1} a 
	&=  \bigl(  ca c^{-1} a \bigr) (ba)^{-(q-1)}
	&& \text{($ba \in T$ is inverted by~$a$)}
	\\&= \bigl(   (wz)a (wz)^{-1} a  \bigr)  (ba)^{-(q-1)}
	\\&= \bigl(   w^2  \bigr)  (ba)^{-(q-1)}
	&& \text{($a$ inverts~$w$ and centralizes~$z$)}
	\\&\neq \bigl(   w^{-2}  \bigr)  (ba)^{-(q-1)}
	\\&= \bigl( b (wz)  b (wz)^{-1} \bigr)  (ba)^{-(q-1)}
	&& \text{($b$ inverts~$w$ and centralizes~$z$)}
	\\&= \bigl( bc  b c^{-1} \bigr)  (ba)^{-(q-1)}
	\\&= bc (ba)^{q-1} b c^{-1}  , 
	&& \text{($ba \in T$ is inverted by~$b$)}
	\end{align*}
these two hamiltonian cycles have different voltages. Therefore at least one of them must have a nontrivial voltage. This nontrivial voltage must generate $\integer_p$, so \cref{FGL} provides a hamiltonian cycle in $\Cay(G;S)$.
\end{proof}

\begin{prop} \label{pin235}
Suppose
	\begin{itemize}
	\item $|G| = 30p $, where $p$~is prime, 
	and
	\item $|G|$ is not square-free {\upshape(}i.e., $p \in \{2,3,5\}${\upshape)}.
	\end{itemize}
Then every Cayley graph on~$G$ has a hamiltonian cycle.
\end{prop}

\begin{proof} 
We know $|G|$ is either $60$, $90$, or~$150$, and it is known that every connected Cayley graph of any of these three orders has a hamiltonian cycle. This can be verified by exhaustive computer search, or see \cite[Props.~7.2 and~9.1]{M2Slovenian-LowOrder} and \cite{GhaderpourMorris-150}. 
\end{proof}

\begin{lem} \label{SquareFree}
Suppose
	\begin{itemize}
	\item $|G| = 30p$, where $p$~is prime, 
	and
	\item $p \ge 7$.
	\end{itemize}
Then 
	\begin{enumerate}
	\item  \label{SquareFree-G'cyclic}
	$G'$ is cyclic,
	\item  \label{SquareFree-G'NotZ(G)}
	$G' \cap Z(G) = \{e\}$,
	\item  \label{SquareFree-ZmxG'}
	$G \iso \integer_n \ltimes G'$, for some $n \in \integer^+$,
	and
	\item  \label{SquareFree-tau}
	if $b$ is a generator of $\integer_n$, and we choose $\tau \in \integer$, such that $x^b = x^\tau$ for all $x \in G'$, then $\gcd \bigl( \tau-1, |a| \bigr) = 1$.
	\end{enumerate}
\end{lem}

\begin{proof}
Since $|G|$ is square-free (because $p \ge 7$), we know that every Sylow subgroup of~$G$ is cyclic. Therefore the conclusions follow from \cite[Thm.~9.4.3, p.~146]{Hall-ThyGrps}\footnote{The condition $[(r-1), nm] = 1$ in the statement of \cite[Cor.~9.4.3, p.~146]{Hall-ThyGrps} suffers from a typographical error --- it should say $\gcd \bigl( (r-1)n, m \bigr) = 1$.}.\refnote{SquareFreeRef}
\end{proof}

\section{Proof of the Main Theorem}

\begin{proof}[\bf Proof of \cref{30p}]
Because of  \cref{pin235}, we may assume 
	$$ p \ge 7 ,$$
so the conclusions of \cref{SquareFree} hold.

We may also assume $|G'|$ is not prime (otherwise \cref{KeatingWitte} applies). Furthermore, if $|G'| = 15p$, then $G$ is a dihedral group,\refnote{15p->dihedralRef}
 so \cref{Dihedral2pqr} applies.
In addition, if $|G'| = 15$, then $G \iso D_{30} \times \integer_p$,\refnote{G'=15Ref}
 so \cref{DihedralxZr} applies.
Thus, we may assume $|G'| = pq$, where $q \in \{3,5\}$.\refnote{G'=pqRef}
 So
	$$ \text{$G = \integer_{2r} \ltimes \integer_{pq}$, with $\{q,r\} = \{3,5\}$ (and $G' = \integer_{pq}$)}  . $$
Note that $\integer_r$ centralizes~$\integer_q$, because there is no nonabelian group of order~$15$, so $\integer_2$ must act nontrivially on~$\integer_q$.\refnote{Z2NontrivialRef}
 Therefore
	$$ \text{$y^x = y^{-1}$ whenever $y \in \integer_q$ and $\langle x \rangle = \integer_{2r}$} .$$
We also assume 
	$$ \text{$\integer_r$ does not centralize~$\integer_p$,} $$
because otherwise $G \iso D_{2pq} \times \integer_r$,\refnote{ZrCentRef}
 so \cref{DihedralxZr} applies. 

Given a minimal generating set~$S$ of~$G$, we may assume 
	$$S \cap G' = \emptyset , $$
for otherwise \fullcref{NormalEasy}{notZ} applies.

\setcounter{case}{0}

\begin{case}
Assume $\#S = 2$.
\end{case}
Write $S = \{a,b\}$.

\begin{subcase}
Assume\/ $|a|$ is odd.
\end{subcase}
This implies $a$ has order~$r$ in $G/G'$, so $(a^{-(r-1)},b^{-1}, a^{r-1},b)$ is a hamiltonian cycle in $\Cay(G/G';S)$. Its voltage is
	$$ a^{-(r-1)} b^{-1} a^{r-1}b = [a^{r-1},b] .$$
Since $\gcd(r-1,|a|) \mid \gcd(r-1,15p) = 1$,\refnote{gcd(r-1bya)Ref}
 we know $\langle a^{r-1},b\rangle = \langle a,b \rangle = G$. So $\langle [a^{r-1},b] \rangle = G'$ \csee{G'2gen}. Therefore \cref{FGL} applies.

\begin{subcase}
Assume $a$ and~$b$ both have even order.
\end{subcase}

\begin{subsubcase}
Assume $a$ has order~$2$ in $G/G'$.
\end{subsubcase}
Note that $q \nmid |a|$, since $\integer_2$ does not centralize~$\integer_q$.\refnote{Z2CentZqRef}
 Also, if $|a| = 2p$, then \cref{DoubleEdge} applies. Therefore, we may assume $|a| = 2$. 

Now $b$ must generate $G/G'$ (since $\langle a, b \rangle = G$, and $b$ has even order), so $b$~has trivial centralizer in~$\integer_{pq}$. Then, since $|a| = 2$ and $\langle a,b \rangle = G$, it follows that $a$ must also have trivial centralizer in~$\integer_{pq}$.\refnote{aTrivCentRef}
Therefore (up to isomorphism), we must have either:\refnote{order2optionsRef}
\begin{enumerate}
\item \label{|a|=2-6}
$a = x^3$ and $b =  xyw$, in $G =  \integer_{6} \ltimes( \integer_5 \times \integer_p) = \langle x \rangle \ltimes \bigl( \langle y \rangle \times \langle w \rangle \bigr)$, with $y^x = y^{-1}$ and $w^x = w^d$, where $d$ is a primitive $6\th$ root of~$1$ in $\integer_p$ (so $d^2 - d + 1 \equiv 0 \pmod{p}$),
or
\item \label{|a|=2-10} 
$a = x^5$ and $b = xyw$, in $G =  \integer_{10} \ltimes ( \integer_3 \times \integer_p ) = \langle x \rangle \ltimes \bigl( \langle y \rangle \times \langle w \rangle \bigr)$ with $y^x = y^{-1}$ and $w^x = w^d$, where $d$ is a primitive $10\th$ root of~$1$ in $\integer_p$ (so $d^4 - d^3 + d^2 - d + 1 \equiv 0 \pmod{p}$).
\end{enumerate}

For \pref{|a|=2-6}, we note that the sequence $ \bigl( (a,b^{-5})^4, a, b^5 \bigr)$ is a hamiltonian cycle in $\Cay(G/\integer_p;S)$:
\begin{align*}
\begin{array}{ccccccccccc}
\qt{e} \bya {x^3} \bybi {x^2y} \bybi {x} \bybi {y} \bybi {x^5}  \\
 \bybi {x^4y}\bya {xy^4} \bybi {y^2} \bybi {x^5y^4} \bybi {x^4y^2}  \\
 \bybi {x^3y^4} \bybi {x^2y^2}\bya {x^5y^3} \bybi {x^4y^3} \bybi {x^3y^3}  \\
 \bybi {x^2y^3} \bybi {xy^3} \bybi {y^3}\bya {x^3y^2} \bybi {x^2y^4}  \\
 \bybi {xy^2} \bybi {y^4} \bybi {x^5y^2} \bybi {x^4y^4} \bya {xy} \\
  \byb {x^2} \byb {x^3y} \byb {x^4} \byb {x^5y} \byb {e} 
. \end{array}
\end{align*}
Calculating modulo the normal subgroup $\langle y \rangle$, its voltage is
\begin{align*}
 (ab^{-5})^4 (ab^5)
 &= (ab)^4(ab^{-1})
 && \text{($b^6 = e$)}
 \\&\equiv \bigl( x^3 \, (xw)  \bigr)^4 \, \bigl( x^3 \, (xw)^{-1} \bigr)
 \\&= \bigl( x^4w  \bigr)^4 \, \bigl( (x w^{-1})^{-1} \, x^3  \bigr)
 && \text{($x^3$ inverts $w$)}
 \\&= \bigl( x^{16} w^{ d^{12} + d^8 + d^4 + 1  }  \bigr) \, \bigl( (wx^{-1}) \, x^3 \bigr)
 \\&= x^{-2}  w^{1 + d^2 -d + 2  }  x^2
 && \begin{pmatrix} \text{$x^6 = e$ and} \\ \text{$d^3 \equiv -1 \pmod{p}$)} \end{pmatrix}
\\&= x^{-2} w^{d^2  +2}  x^2
\\&= x^{-2} w^{d  +1}  x^2
 && \bigl(d^2 - d + 1 \equiv 0 \pmod{p} \bigr)
, \end{align*}
which is nontrivial. Therefore, the voltage generates~$\integer_p$, so \cref{FGL} provides a hamiltonian cycle in $\Cay(G;S)$.

For \pref{|a|=2-10}, here is a hamiltonian cycle in $\Cay(G/\integer_p;S)$:
\begin{align*}
\begin{array}{ccccccccccc}
\qt{e} \bya {x^5} \byb {x^6y} \byb {x^7} \byb {x^8y} \byb {x^9} \\
\bya{x^4} \byb {x^5y} \bya {y^2} \byb {xy^2} \byb {x^2y^2}  \\
\byb {x^3y^2} \byb{x^4y^2} \bya {x^9y} \bybi {x^8} \bybi {x^7y} \\
 \bybi {x^6} \bya {x} \bybi {y} \bya {x^5y^2} \byb {x^6y^2}  \\
\byb {x^7y^2} \bya {x^2y} \byb {x^3} \byb{x^4y} \bya {x^9y^2} \\
\bybi {x^8y^2} \bya {x^3y} \bybi {x^2} \bybi {xy} \bybi{e}
 . \end{array}
\end{align*}
Calculating modulo~$\langle y \rangle$, its voltage is 
\begin{align*}
a b^4 (a b a) b^4 & (a b^{-3} a) b^{-1} (a b^2)^2 (a b^{-1} a) b^{-3}
\\&\equiv x^5  (xw)^4 \bigl(  x^5   (xw)   x^5 \bigr)  (xw)^4  \bigl(x^5 (xw)^{-3}   x^5 \bigr)
\\& \hskip 0.75in
  \cdot  (xw)^{-1} \bigl( x^5  (xw)^2 \bigr)^2  \bigl( x^5  (xw)^{-1}  x^5 \bigr)  (xw)^{-3}
\\&= x^5  (xw)^4 \bigl(  xw^{-1} \bigr)  (xw)^4  \bigl( xw^{-1} \bigr)^{-3} 
\\& \hskip 0.75in
  \cdot  (xw)^{-1} \bigl( (x w^{-1})^2 (xw)^2 \bigr)  \bigl( xw^{-1} \bigr)^{-1} (xw)^{-3}
\\&= x^5  (x^4w^{d^3 + d^2 + d + 1}) \bigl(  xw^{-1} \bigr)   (x^4w^{d^3 + d^2 + d + 1})   \bigl( w^{d^2 + d + 1} x^{-3} \bigr)
\\& \hskip 0.75in
  \cdot (w^{-1} x^{-1}) \bigl( x^4 w^{- d^3 - d^2 + d + 1}  \bigr)  \bigl( wx^{-1} \bigr) (w^{-(d^2 + d + 1)} x^{-3})
\\&= w^{d(d^3 + d^2 + d + 1)}w^{-1} w^{d^6(d^3 + d^2 + d + 1)} w^{d^6(d^2 + d + 1)}
  \\& \hskip 0.75in
  \cdot w^{-d^9} w^{d^6(- d^3 - d^2 + d + 1)}  w^{d^6} w^{-d^7(d^2 + d + 1)}
  \\&= w^{-2d^9  + 2d^7 + 4d^6 + d^4 + d^3 + d^2 + d - 1}
. \end{align*}
Modulo~$p$, the exponent of~$w$ is:
	\begin{align*}
	-2d^9 & {}  +  2d^7 + 4d^6 + d^4 + d^3 + d^2 + d - 1
	\\&\equiv 2d^4  - 2d^2 - 4d + d^4 + d^3 + d^2 + d - 1
  	&& \text{(because $d^5 \equiv -1$)}
	\\&= 3d^4 + d^3  - d^2 - 3d - 1
	\\&= 3(d^4 - d^3 + d^2 - d + 1) + 4(d^3 - d^2  - 1)
	\\& \equiv 3(0)+ 4(d^3 - d^2  - 1)
	\\&= 4(d^3 - d^2  - 1)
	. \end{align*}
This is nonzero (mod~$p$), because $d^4 - d^3 + d^2 - d + 1 \equiv 0 \pmod{p}$ and
	$$ (d^3 - d^2)(d^3 - d^2  - 1) - (d^2 - d - 1)(d^4 - d^3 + d^2 - d + 1) = 1 .$$
Therefore the voltage generates $\langle w \rangle = \integer_p$, so \cref{FGL} applies.

\begin{subsubcase}
Assume $a$ and~$b$ both have order~$2r$ in $G/G'$.
\end{subsubcase}
Then $|a| = |b| = 2r$ (because $\integer_{2r}$ has trivial centralizer in~$\integer_{pq}$).\refnote{2rTrivCentRef}

We have $a \in b^i G'$ for some $i$ with $\gcd(i,2r) = 1$. We may assume $1 \le i < r$ by replacing $a$ with its inverse if necessary. Here is a hamiltonian cycle in $\Cay(G/G';S)$:\refnote{a=biRef}
	$$ \bigl( (a,b,a^{-1},b)^{(i-1)/2},a, b^{2r+1-2i} \bigr) .	$$
To calculate its voltage, write $a = b^i yw$, where $\langle y \rangle = \integer_q$ and $\langle w \rangle = \integer_p$.
We have $y^b = y^{-1}$ and $w^b = w^d$, where $d$ is a primitive $r\th$ or $(2r)\th$ root of unity\refnote{ror2rthrootRef}
 in~$\integer_p$. Then the voltage of the walk is:
	\begin{align*}
	(aba^{-1}b)^{(i-1)/2}a b^{2r+1-2i}
	&= \bigl( (b^i yw) b(b^i yw)^{-1}b \bigr)^{(i-1)/2}(b^i yw) b^{1-2i}
	\\&= \bigl( (b^i yw) b(w^{-1} y^{-1} b^{-i})b \bigr)^{(i-1)/2}(b^i yw) b^{1-2i}
	\\&= \bigl(  b^2    y^{-2}   w^{(d-1)d^{1-i}}\bigr)^{(i-1)/2}(b^i yw) b^{1-2i}
	&&\hbox to 0pt{\color{blue}\ (note \ref{ainbiG'-1Ref})\hss} 
	\\&= \bigl( b^{i-1} y^{-(i-1)}  w^{(d-1)d^{1-i}(d^{i-3} + d^{i-5} + \cdots + d^2 + 1)}\bigr) (b^i yw) b^{1-2i}
	&&\hbox to 0pt{\color{blue}\ (note \ref{ainbiG'-2Ref})\hss} 
	\\&=  b^{2i-1} y^{(i-1)+1}  w^{(d-1)d(d^{i-3} + d^{i-5} + \cdots + d^2 + 1)+1} b^{1-2i}
	.
	&&\hbox to 0pt{\color{blue}\ (note \ref{ainbiG'-3Ref})\hss} 
\end{align*}
Now:
	\begin{itemize}
	\item The exponent of~$y$ is $(i-1) + 1 = i$. If $q \mid i$, then, since $i < r$, we must have $q = 3$, $r = 5$, and $i = 3$.\refnote{q=3&r=5&i=3Ref}
	\item The exponent of~$w$ is
		\begin{align*}
		(d-1) & d(d^{i-3} + d^{i-5} + \cdots + d^2 + 1)+1
		= d(d-1)\frac{d^{i-1} - 1}{d^2 - 1}+1
		\\&= d\frac{d^{i-1} - 1}{d+1}+1
		= \frac{d^{i} - d}{d+1}+\frac{d+1}{d+1}
		= \frac{d^{i} + 1}{d+1}
		.\end{align*}
	This is not divisible by~$p$, because $d$ is a primitive $r\th$ or $(2r)\th$ root of~$1$ in~$\integer_p$, and $\gcd(i,2r) = 1$.
	\end{itemize}
Thus, the voltage generates~$G'$ (so \cref{FGL} applies) unless $q = 3$, $r = 5$, and $i = 3$. 

In this case, since $i = 3$, we have $a = b^3 y w$. Also, we may assume $b = x$. Then a hamiltonian cycle in $\Cay(G/\integer_p;S)$ is:
\begin{align*}
\begin{array}{ccccccccccc}
\qt{e} \byai {x^7y} \byai {x^4} \byai {xy} \byai {x^8} \byai {x^5y}  \\
 \byai {x^2}\byai {x^9y} \byai {x^6} \byai {x^3y} \byb {x^4y^2}  \\
\bya {x^7y^2} \bya {y^2} \bya {x^3y^2} \bya {x^6y^2} \bya {x^9y^2}  \\
\bya {x^2y^2} \bya {x^5y^2} \bya {x^8y^2} \bya {xy^2} \byb {x^2y}  \\
\bya {x^5} \bya {x^8y} \bya {x} \bya {x^4y} \bya {x^7} \\
 \bya {y} \bya {x^3} \bya {x^6y} \bya {x^9} \byb {e} 
. \end{array}
\end{align*}
Calculating modulo~$\langle y \rangle$, and noting that $|a| = 2r = 10$, its voltage is 
\begin{align*}
a^{-9}b(a^9b)^2
&= ab (a^{-1}b)^2
\equiv \bigl( (x^3w) x \bigr) \bigl( w^{-1} x^{-2} \bigr)^2
\\&= \bigl( x^4 w^d \bigr)\bigl( w^{-1 - d^2}  x^{-4} \bigr)
 = x^4 w^{-(d^2 - d + 1)} x^{-4}
. \end{align*}
Since $d$ is a primitive $5\th$ or $10\th$ root of~$1$ in~$\integer_p$, we know that it is not a primitive $6\th$ root of~$1$, so $d^2 - d + 1 \not\equiv 0 \pmod{p}$. Therefore the voltage is nontrivial, and hence generates~$\integer_p$, so \cref{FGL} applies.

\begin{case} \label{S=3Min}
Assume $\#S = 3$, and $S$ remains minimal in $G/\integer_p = \qt{G}$.
\end{case}
Since $G = \integer_{2r} \ltimes \integer_{pq}$ and $\integer_r$ centralizes~$\integer_q$, we know $\qt{G} \iso (\integer_2 \ltimes \integer_q) \times \integer_r$. Also, since $\integer_2$ inverts $\integer_q$, we have $\integer_2 \ltimes \integer_q \iso D_{2q}$. Therefore, $\qt{G} \iso D_{2q} \times \integer_r$, so
 we may write $S = \{a,b,c\}$ with $\langle \qt{a},\qt{b} \rangle = D_{2q}$ and $\langle \qt{c} \rangle = \integer_r$.\refnote{S=3-abcRef}
  Since $S \cap G' = \emptyset$, we know that $\qt{a}$ and~$\qt{b}$ are reflections, so they have order~$2$ in $G/ \integer_p$. Therefore, we may assume $|a| = |b| = 2$, for otherwise \cref{DoubleEdge} applies. Also, since $\integer_r$ does not centralize $\integer_p$, we know that $|c| = r$.\refnote{S=3-c=rRef}
   Replacing $c$ by a conjugate, we may assume $\langle c \rangle = \integer_r$.

We may assume $\integer_r \not\subset Z(G)$ (otherwise \cref{DihedralxZr} applies),\refnote{DihedralxZrRef}
 so we may assume $[a,c] \neq e$ (by interchanging $a$ and~$b$ if necessary). Let 
	$$ W = \bigl( (b,a)^{q-1}, c, (c^{r-2}, a, c^{-(r-2)}, b)^{q-1} \bigr) .$$
Then
	$$ \text{$\bigl( W, c^{r-2}, a,c^{-(r-1)}, a \bigr)$ \ and \ $\bigl( W, c^{r-3}, a,c^{-(r-1)}, a, c \bigr)$} $$
are hamiltonian cycles in $\Cay(G/G'; S)$.\refnote{WcyclesRef}
Let $v$ be the voltage of the first of these, and let $\gamma = [a,c] \,  [a,c]^{ac}$. Then the voltage of the second is
	\begin{align*}
	v \cdot (c^{r-2} ac^{-(r-1)} a)^{-1} (c^{r-3} a c^{-(r-1)} a c) 
	&= v \cdot \bigl(a c^{r-1} ac^{-(r-2)}) (c^{r-3} a c^{-(r-1)} a c \bigr) 
	\\&= v \cdot \bigl(a c^{-1} ac^{-1} a c a c \bigr) 
	\\&= v \cdot \bigl(a c^{-1} [a,c] a c \bigr) 
	\\&= v \cdot \bigl(a c^{-1} ac [a,c]^{ac} \bigr) 
	\\&= v \cdot \bigl([a,c] \,  [a,c]^{ac} \bigr) 
	\\&= v \gamma 
	. \end{align*}
Since $[a,c]$ generates $\integer_p$, and $ac$ does not invert~$\integer_p$ (this is because $a$ inverts~$\integer_p$, and $c$~does not centralize~$\integer_p$), we know $\gamma \neq e$. Therefore $v$ and~$v \gamma$ cannot both be trivial, so at least one of them generates~$\integer_p$. Then \cref{FGL} provides a hamiltonian cycle in $\Cay(G;S)$.

\begin{case}
Assume $\#S = 3$, and $S$ does not remain minimal in $G/\integer_p$.
\end{case}
Choose a $2$-element subset $\{a,b\}$ of~$S$ that generates $G/\integer_p$.
As in \cref{S=3Min}, we have $G/\integer_p \iso D_{2q} \times \integer_r$.
From the minimality of~$S$, we see that $\langle a,b \rangle = D_{2q} \times \integer_r$ (up to a conjugate).\refnote{<ab>=DxZRef}
The projection of~$\{a,b\}$ to $D_{2q}$ must be of the form $\{f,y\}$ or $\{f, fy\}$, where $f$~is a reflection and $y$~is a rotation. Thus, using $z$ to denote a generator of~$\integer_r$ (and noting that $y \notin S$, because $S \cap G' = \emptyset$), we see that $\{a,b\}$~must be of the form\refnote{fyRef}
	\begin{enumerate}
	\item $\{f, yz\}$,
	or
	\item $\{f, fyz\}$,
	or
	\item $\{fz, yz^\ell\}$, with $\ell \not\equiv 0 \pmod{r}$,
	or
	\item $\{fz, fyz^\ell\}$, with $\ell \not\equiv 0 \pmod{r}$.
	\end{enumerate}
Let $c$ be the final element of~$S$. We may write 
	$$ \text{$c = f^i y^j z^k w$
	\quad with \ $0 \le i < 2$, \ $0 \le j < q$, \ and \  $0 \le k < r$}
	 .$$
Note that, since $S \cap G' = \emptyset$, we know that  $i$ and~$k$ cannot both be~$0$.
Let $d$ be a primitive $r\th$ root of unity in~$\integer_p$, such that 
	$$ \text{$w^z = w^d$ for $w \in \integer_p$.} $$

\begin{subcase}
Assume $a = f$ and $b = yz$.
\end{subcase}
From the minimality of~$S$, we know $\langle b,c \rangle \neq G$, so $i=0$,\refnote{a=f&b=yz-i=0Ref} 
so we must have $k \neq 0$. 

\begin{subsubcase}
Assume $k = 1$.
\end{subsubcase}
Then $b \equiv c \pmod{G'}$, so we have the hamiltonian cycles $(a,b^{-(r-1)},a,b^{r-2},c)$ and  $(a,b^{-(r-1)},a,b^{r-3},c^2)$ in $\Cay(G/G';S)$. The voltage of the first is
	\begin{align*}
	 ab^{-(r-1)}a b^{r-2}c
	&= \bigl( ab^{-(r-1)}ab^{r-1} \bigr) \bigl( b^{-1}c \bigr)
	\\&= \bigl( (f) (yz)^{-(r-1)}(f) (yz)^{r-1} \bigr) \bigl( (yz)^{-1}(y^j zw) \bigr)
	\\&= \bigl( y^{2(r-1)} \bigr) \bigl( y^{j-1}w\bigr)
		&&\hbox to 0pt{\color{blue}\ (note \ref{a=f&b=yz-k=1aRef})\hss} 
	\\&= \begin{cases}
	y^{j+3}w & \text{if $r = 3$ and $q = 5$}, \\
	y^{j+7}w & \text{if $r = 5$ and $q = 3$} 
	\end{cases}
	\\&= y^{j-2} w
	,
		&&\hbox to 0pt{\color{blue}\ (note \ref{a=f&b=yz-k=1bRef})\hss} 
	 \end{align*}
which generates $\integer_q \times \integer_p  = G'$ if $j \neq 2$. 

So we may assume $j = 2$ (for otherwise \cref{FGL} applies). In this case, the voltage of the second hamiltonian cycle is
	\begin{align*}
	 ab^{-(r-1)}a b^{r-3}c^2
	&= \bigl( ab^{-(r-1)}ab^{r-1} \bigr) \bigl( b^{-2}c^2 \bigr)
	\\&= \bigl( (f) (yz)^{-(r-1)}(f) (yz)^{r-1} \bigr) \bigl( (yz)^{-2}(y^2 zw)^2 \bigr)
	\\&= \bigl( y^{2(r-1)} \bigr) \bigl( y^{2}w^{d+1}\bigr)
		&&\hbox to 0pt{\color{blue}\ (note \ref{a=f&b=yz-k=1jaRef})\hss} 
	\\&= \begin{cases}
	y^{6}w^{d+1} & \text{if $r = 3$ and $q = 5$}, \\
	y^{10}w^{d+1} & \text{if $r = 5$ and $q = 3$} 
	\end{cases}
	\\&= y w^{d+1}
	,
		&&\hbox to 0pt{\color{blue}\ (note \ref{a=f&b=yz-k=1jbRef})\hss} 
	 \end{align*}
which generates $\integer_q \times \integer_p  = G'$.\refnote{a=f&b=yz-k=1jcRef}
 So \cref{FGL} provides a hamiltonian cycle in $\Cay(G;S)$.
 
 \begin{subsubcase}
 Assume $k > 1$.
 \end{subsubcase}
 We may replace~$c$ with its inverse, so we may assume $k \le (r-1)/2$. Therefore $r \neq 3$, so we must have $r = 5$ and $k = 2$. So $a = f$, $b = yz$, and $c = y^jz^2w$.

\begin{subsubsubcase}
Assume $j = 0$.
\end{subsubsubcase}
Here is a hamiltonian cycle in $\Cay(G/\integer_p;S)$:
\begin{align*}
\begin{array}{ccccccccccc}
\qt{e} \bya {f} \byb {fyz} \bya {y^2z} \byb {z^2} \bya {fz^2} \\
 \byb {fyz^3} \bya {y^2z^3} \byb {z^4} \bya {fz^4} \bybi {fy^2z^3} \\
 \bya {yz^3} \byb {y^2z^4} \byci {y^2z^2} \bya {fyz^2} \byc {fyz^4}  \\
\bybi {fz^3} \bya {z^3} \byb {yz^4} \bya {fy^2z^4} \byci {fy^2z^2}  \\
\bya {yz^2} \byci {y} \bya {fy^2} \byb {fz} \bya {z} \\
\bybi {y^2} \bya {fy} \byb {fy^2z} \bya {yz} \bybi {e}  
. \end{array}
\end{align*}
Letting $\epsilon \in \{\pm1\}$, such that  $w^f = w^\epsilon $, and calculating modulo~$\langle y \rangle$, its voltage is 
\begin{align*}
(ab)^4 & (a b^{-1} a b) (c^{-1} a c) (b^{-1} ab)  (a c^{-1})^2 (a b a b^{-1})^2
\\&\equiv (fz)^4 (f z^{-1} fz) (w^{-1}z^{-2} f z^2 w) (z^{-1} f z)  (f w^{-1} z^{-2})^2 (fzf z^{-1} )^2
\\&= (z^4) (e ) (w^{\epsilon -1} f) (f)  ( w^{-(\epsilon + d^2)} z^{-4}) ( e)
		&&\hbox to 0pt{\color{blue}\ (note \ref{a=f&b=yz-k>1Ref})\hss} 
\\&= z^4 w^{-(d^2 + 1)} z^{-4}
. \end{align*}
Since $d$ is a primitive $5\th$ root of unity in $\integer_p$, we know that $d^2 + 1 \not\equiv 0 \pmod{p}$, so the voltage is nontrivial, and hence generates $\integer_p$, so \cref{FGL} applies.

\begin{subsubsubcase}
Assume $j \neq 0$.
\end{subsubsubcase}
Since $\langle a,c \rangle \neq G$, this implies $f$ centralizes $\integer_p$, so $G = D_{6} \times (\integer_5 \ltimes \integer_p)$.\refnote{jnot0Ref}

If $j = 1$ (so $c = y z^2 w$), here is a hamiltonian cycle in $\Cay(G/\integer_p;S)$:
\begin{align*}
\begin{array}{ccccccccccc}
\qt{e} \bya {f} \byb {fyz} \bya {y^2z} \byb {z^2} \bya {fz^2}  \\
\byb {fyz^3} \bya {y^2z^3} \byb {z^4} \byb {y} \bya {fy^2}  \\
\byb {fz} \bya {z} \bybi {y^2} \bya {fy} \byb {fy^2z} \\
 \bya {yz} \byb {y^2z^2} \bya {fyz^2} \byc {fy^2z^4}\bya {yz^4}   \\
\bybi {z^3} \bya {fz^3} \byb {fyz^4} \bya {y^2z^4} \bybi {yz^3} \\
\bya {fy^2z^3} \byb {fz^4} \byci {fy^2z^2} \bya {yz^2} \byci {e} 
. \end{array}
\end{align*}
Calculating modulo the normal subgroup $D_6 = \langle f, y \rangle$, its voltage is 
\begin{align*}
(ab)^4 & (ba)^2 (b^{-1} a) (ba)^2 (c) (ab^{-1} ab)^2 (c^{-1} a c^{-1})
\\&\equiv (ez)^4 (ze)^2 (z^{-1} e) (ze)^2 (z^2 w) (ez^{-1} ez)^2 (w^{-1} z^{-2} e w^{-1} z^{-2})
\\&= z^7 w^{-1} z^{-2}
\\&= z^2 w^{-1} z^{-2}
,\end{align*}
because $|z| = r = 5$. Since this voltage generates $\integer_p$, \cref{FGL} provides a hamiltonian cycle in $\Cay(G;S)$.

If $j = 2$ (so $c = y^2 z^2 w$), here is a hamiltonian cycle in $\Cay(G/\integer_p;S)$:
\begin{align*}
\begin{array}{ccccccccccc}
\qt{e} \bybi {y^2z^4} \bya {fyz^4} \byb {fy^2} \byb {fz} \bya {z}  \\
\byb {yz^2} \bya {fy^2z^2} \byb {fz^3} \bya {z^3} \byc {y^2}  \\
\bybi {yz^4} \bya {fy^2z^4} \byb {f} \byb {fyz} \bya {y^2z}  \\
\byb {z^2} \bya {fz^2} \byb {fyz^3} \bya {y^2z^3} \byc {y} \\
 \bybi {z^4} \bya {fz^4} \byb {fy} \byb {fy^2z} \bya {yz} \\
 \byb {y^2z^2} \bya {fyz^2} \byb {fy^2z^3} \bya {yz^3} \byc {e} 
. \end{array}
\end{align*}
Calculating modulo the normal subgroup $D_6 = \langle f, y \rangle$, its voltage is 
\begin{align*}
\bigl( b^{-1} ab^2 (ab)^2(ac) \bigr)^3
\equiv \bigl( z^{-1} e z^2 (ez)^2(ez^2w) \bigr)^3
= \bigl( z^5w \bigr)^3
= w^3
,\end{align*}
because $|z| = r = 5$. Since this voltage generates $\integer_p$, \cref{FGL} provides a hamiltonian cycle in $\Cay(G;S)$.

\begin{subcase}
Assume $a = f$ and $b = fyz$.
\end{subcase}
Since $\langle b, c \rangle \neq G$, we must have $c \in \langle fy, z \rangle w$,\refnote{cInfyzRef}
 so 
	$$ \text{$c = (fy)^iz^k w$ \quad with \  $0 \le i < 2$ \ and \ $0 \le k < r$.} $$

\begin{subsubcase}
Assume $k = 0$.
\end{subsubcase}
Then  $c = fyw$, so we have $c \equiv a \pmod{G'}$. Therefore $(b^{-(r-1)}, a, b^{r-1}, c)$ is a hamiltonian cycle in $\Cay(G/G';S)$. Since 
	$$b^{r-1} = (fyz)^{r-1} =  (fy)^{r-1}(z^{r-1}) = (e)(z^{-1}) = z^{-1}
	 ,
	 \hbox to 0pt{\color{blue}\hskip 1in(note \ref{3.2.1b^r-1Ref})\hss} 
	$$
its voltage is
	$$ b^{-(r-1)} a b^{r-1} c
	= (b^{-(r-1)} a b^{r-1} a)(a c)
	= [b^{r-1}, a](a c)
	= [ z^{-1}, f] (yw)
	= yw ,$$
which generates $\integer_q \times \integer_p = G'$, so \cref{FGL} provides a hamiltonian cycle in $\Cay(G;S)$.

\begin{subsubcase}
Assume $i = 0$.
\end{subsubcase}
Then $c = z^k w$, and we know $k \neq 0$, because $S \cap G' = \emptyset$.

If $k = 1$, then $\bigl( (a,c)^{r-1}, a, b \bigr)$ is a hamiltonian cycle in $\Cay(G/G';S)$.\refnote{3.3.2HamRef}
Letting $\epsilon \in \{\pm 1\}$, such that $w^f = w^\epsilon $, its voltage is
	\begin{align*}
	(ac)^{r-1} \, a \, b
	&= (ac)^{r} \, (c^{-1} \, b)
			&&\hbox to 0pt{\color{blue}\ (note \ref{3.2.2aRef})\hss} 
	\\&= (fzw)^{r} \bigl( (zw)^{-1} (fyz) \bigr)
	\\&= (f^r z^r w^{(\epsilon d)^{r-1} +(\epsilon d)^{r-2} + \cdots + 1})  \bigl( w^{-1} z^{-1} fyz \bigr)
			&&\hbox to 0pt{\color{blue}\ (note \ref{3.2.2bRef})\hss} 
	\\&= f \, w^{(\epsilon d)^{r-1} +(\epsilon d)^{r-2} + \cdots +  \epsilon d} \,  fy 
			&&\hbox to 0pt{\color{blue}\ (note \ref{3.2.2cRef})\hss} 
	\\&= w^{\epsilon ( (\epsilon d)^{r-1} +(\epsilon d)^{r-2} + \cdots + \epsilon d )} y
	\\&= w^{d( (\epsilon d)^{r-2} +(\epsilon d)^{r-3} + \cdots + 1 )} y
	.\end{align*}
Since $\epsilon d$ is a primitive $r\th$ or $(2r)\th$ root of unity in $\integer_p$, it is clear that the exponent of~$w$ is nonzero (mod~$p$).\refnote{3.2.2dRef}
 Therefore the voltage generates $\integer_p \times \integer_q = G'$, so \cref{FGL} provides a hamiltonian cycle in $\Cay(G;S)$.

We may now assume $k \ge 2$. However, we may also assume $k \le (r-1)/2$ (by replacing $c$ with its inverse if necessary). So $r = 5$ and $k = 2$. In this case, here is a hamiltonian cycle in $\Cay(G/\integer_p;S)$:
\begin{align*}
\begin{array}{ccccccccccc}
\qt{e} \bya {f} \byb {fyz} \bya {y^2z} \bybi {y} \bya {fy^2} \\
\byb {fz} \bya {z} \bybi {y^2} \bya {fy} \byb {fy^2z} \\
\bya {yz} \byb {y^2z^2} \bya {fyz^2} \byb {fy^2z^3} \bya {yz^3} \\
\byb {y^2z^4} \bya {fyz^4} \bybi {fz^3} \bya {z^3} \byb {yz^4} \\ 
\byci {yz^2} \bya {fy^2z^2} \byc {fy^2z^4} \bybi {fyz^3}  \bya {y^2z^3} \\
\byb {z^4} \bya {fz^4} \byci {fz^2} \bya {z^2} \byci {e}
.\end{array}
\end{align*}
Its voltage is 
	\begin{align*}
	(abab^{-1})^2 (ab)^4 (ab^{-1}ab)(c^{-1}ac)(b^{-1}ab)(ac^{-1})^2 
	. \end{align*}
Since the voltage is in $\integer_p$, it is a power of~$w$, and it is clear that the only terms that contribute a power of~$w$ to the product are contained in the last three parenthesized expressions (because $c$ does not appear anywhere else). Choosing $\epsilon \in \{\pm1\}$, such that $w^f = w^\epsilon $, we calculate the product of these three expressions modulo~$\langle y \rangle$: 
	\begin{align*}
	(c^{-1}ac)(b^{-1}ab)(ac^{-1})^2 
	&\equiv \bigl( (z^2 w)^{-1} f (z^2 w) \bigr) \bigl( (fz)^{-1}f (fz) \bigr) \bigl(f(z^2 w)^{-1} \bigr)^2
	\\&= \bigl( w^{\epsilon - 1} f \bigr) \bigl( f \bigr) \bigl(w^{-(\epsilon + d^2)} z^{-4} \bigr)
			&&\hbox to 0pt{\color{blue}\ (note \ref{3.2.2k>1Ref})\hss} 
	\\&= w^{-(d^2 + 1)} z^{-4}
	\end{align*}
Since the power of~$w$ is nonzero, the voltage generates~$\integer_p$, so \cref{FGL} provides a hamiltonian cycle in $\Cay(G;S)$.

\begin{subsubcase}
Assume $i$ and~$k$ are both nonzero.
\end{subsubcase}
Since $\langle a,c \rangle \neq G$, this implies that  $f$ centralizes~$w$.\refnote{3.2.3CentRef}
Therefore $G = D_{2q} \times (\integer_r \ltimes \integer_p)$.
 Also, since $0 \le i < 2$, we know $i = 1$, so $c = fyz^k w$. We may assume $k \neq 1$ (for otherwise $b \equiv c \pmod{\integer_p}$, so \cref{DoubleEdge} applies). Since we may also assume that $k \le (r-1)/2$ (by replacing $c$ with its inverse if necessary), then we have $r = 5$ and $k = 2$.

Here is a hamiltonian cycle in $\Cay(G/\integer_p;S)$:
\begin{align*}
\begin{array}{ccccccccccc}
\qt{e} \bya {f} \byb {yz} \bya {fy^2z} \byb {y^2z^2} \bya {fyz^2} \\
 \byc{z^4} \bya {fz^4} \bybi {yz^3} \bya {fy^2z^3} \byc {y^2} \\
\bya {fy} \byb {z} \bya {fz} \byb {yz^2} \bya {fy^2z^2}  \\
\byc {y^2z^4} \bya {fyz^4} \bybi {z^3} \bya {fz^3} \byc {y}  \\
\bya {fy^2} \byb {y^2z} \bya {fyz} \byb {z^2} \bya {fz^2} \\
\byc {yz^4} \bya {fy^2z^4} \bybi {y^2z^3} \bya {fyz^3} \byc {e}
.\end{array}
\end{align*}
Calculating modulo the normal subgroup $D_6 = \langle f, y \rangle$, its voltage is 
	\begin{align*}
	\bigl( (ab)^2 a c a b^{-1} a c \bigr)^3 
	&\equiv \bigl( (ez)^2 e (z^2w) e z^{-1} e (z^2w) \bigr))^3 
	\\&= \bigl( z^4 w z w \bigr)^3 
	\\&= w^{3(d+1)}
	,
		&&\hbox to 0pt{\color{blue}\ (note \ref{3.2.3VoltRef})\hss} 
\end{align*}
which generates $\langle w \rangle = \integer_p$,\refnote{3.2.3GenRef}
 so \cref{FGL} applies.

\begin{subcase}
Assume $a = fz$ and $b = yz^\ell$, with $\ell \neq 0$.
\end{subcase}
Since $\langle a, c \rangle \neq G$ and $\langle b, c \rangle \neq G$, we must have $c \in \langle f, z \rangle w$ and $c \in \langle y, z \rangle w$.\refnote{3.3GRef}
 So $c \in \langle z \rangle w$; write $c = z^k w$ (with $k \neq 0$, because $S \cap G' = \emptyset$). 

\begin{subsubcase}
Assume $\ell = k$.
\end{subsubcase}
Then $b \equiv c \equiv z^\ell \pmod{G'}$, so
	$$\bigl( a^{-1}, b^{-(r-1)},a,b^{r-2}, c \bigr) $$
is a hamiltonian cycle in $\Cay(G/G';S)$. Its voltage is
	\begin{align*}
	 a^{-1} b^{-(r-1)}ab^{r-2} c
	 &= (fz)^{-1} (y z^\ell)^{-(r-1)} (fz)(yz^\ell)^{r-2} (z^\ell w)
	 \\&= (f^{-1}y^{-(r-1)}f)y^{r-2}w
	 && \begin{pmatrix} \text{$z$ commutes} \\ \text{with $f$ and~$y$} \end{pmatrix}
	 \\&= (y^{r-1})y^{r-2}w
	 && \text{($f$ inverts~$y$)}
	 \\&= y^{2r-3}w
	. \end{align*}
Since $2(3) - 3 \not \equiv 0 \pmod{5}$ and $2(5) - 3 \not\equiv 0 \pmod{3}$, we have $2r-3 \not\equiv 0 \pmod{q}$, so $y^{2r-3}$ is nontrivial, and hence generates~$\integer_q$. Therefore, this voltage generates $\integer_q \times \integer_p = G'$. So \cref{FGL} provides a hamiltonian cycle in $\Cay(G;S)$.

\begin{subsubcase}
Assume $\ell \neq k$.
\end{subsubcase}
We may assume $\ell,k \le (r-1)/2$ (perhaps after replacing $b$ and/or~$c$ by their inverses). Then we must have $r = 5$ and $\{\ell,k\} = \{1,2\}$.\refnote{3.3.2r=5Ref}

For $(\ell,k) = (1,2)$, here is a hamiltonian cycle in $\Cay(G/\integer_p;S)$:
\begin{align*}
\begin{array}{ccccccccccc}
\qt{e} \bya {fz} \byb {fyz^2} \byai {y^2z} \byai {fy} \bybi {fz^4} \\
\byai {z^3} \byai {fz^2} \byai {z} \byai {f} \bybi {fy^2z^4} \\
\bya {y} \bya {fy^2z} \bya {yz^2} \bya {fy^2z^3} \bya {yz^4} \\
\bya {fy^2} \bya {yz} \bya {fy^2z^2} \bya {yz^3} \byb {y^2z^4} \\
\byai {fyz^3} \byai {y^2z^2} \byai {fyz} \byai {y^2}  \byai {fyz^4} 
\\ \byai {y^2z^3} \byb {z^4} \byai {fz^3} \byai {z^2} \byci {e}  
.\end{array}
\end{align*}
Its voltage is
	$$ a b a^{-2} b^{-1}a^{-4}b^{-1}a^9ba^{-6}ba^{-2}c^{-1} . $$
Since there is precisely one occurrence of~$c$ in this product, and therefore only one occurrence of~$w$, it is impossible for this appearance of~$w$ to cancel.  So the voltage is nontrivial, and therefore generates~$\integer_p$, so \cref{FGL} provides a hamiltonian cycle in $\Cay(G;S)$.

For $(\ell,k) = (2,1)$, here is a hamiltonian cycle in $\Cay(G/\integer_p;S)$:
\begin{align*}
\begin{array}{ccccccccccc}
\qt{e} \byai {fz^4} \byai {z^3} \byai {fz^2} \byai {z} \byai {f} \\
\byai {z^4} \byb {yz} \byai {fy^2} \byai {yz^4} \byc {y} \\
\byai {fy^2z^4} \byai {yz^3} \byai {fy^2z^2} \byc {fy^2z^3} \byai {yz^2} \\
\byai {fy^2z} \byb {fz^3} \byai {z^2} \byai {fz} \byb {fyz^3} \\
\byai {y^2z^2} \byai {fyz} \byc {fyz^2} \byai {y^2z} \byai {fy} \\
\byai {y^2z^4} \byc {y^2} \byai {fyz^4} \byai {y^2z^3} \byb {e} 
.\end{array}
\end{align*}
Choosing $\epsilon \in \{\pm1\}$, such that $w^f = w^\epsilon $, we calculate the voltage, modulo~$\langle y \rangle$: 
\begin{align*}
a^{-4} &\Bigl( \bigl( a^{-2}ba^{-2} \bigr) c a^{-3}c \bigl( a^{-2}b \bigr) \Bigr)^2
\\&\equiv (fz)^{-4} \Bigl( \bigl( (fz)^{-2} z^2 (fz)^{-2} \bigr) (zw) (fz)^{-3} (zw) \bigl( (fz)^{-2}z^2 \bigr) \Bigr)^2
\\&= z^{-4} \bigl( (z^{-2}) (zw) (fz^{-3}) (zw) (e) \bigr)^2
	&&\hbox to 0pt{\color{blue}\ (note \ref{3.3.2fzRef})\hss} 
\\&= z^{-4} \bigl( z^{-1} w fz^{-2}w \bigr)^2
\\&= z^{-4} (w^{d^6+ \epsilon d^4 + \epsilon d^3+ d} z^{-6})
	&&\hbox to 0pt{\color{blue}\ (note \ref{3.3.2zRef})\hss} 
\\&= z^{-4} (w^{d( \epsilon d^3 + \epsilon d^2 +2)} z^4)
.
	&&\hbox to 0pt{\color{blue}\ (note \ref{3.3.2z4Ref})\hss} 
\end{align*}
Since $d$ is a primitive $r\th$ root of unity in~$\integer_p$, and $r = 5$, we know $d^4 + d^3 + d^2 + d + 1 \equiv 0 \pmod{5}$. Combining this with the fact that
	$$ -(d^3 + d^2 - 1)(d^3 + d^2 + 2) + (d^2 + d - 1) (d^4 + d^3 + d^2 + d + 1) = 1,  $$
and
	$$ ( d^3 + d^2 + 3)(-d^3 + -d^2 + 2) + (d^2 + d - 1) (d^4 + d^3 + d^2 + d + 1) = 5 \not\equiv 0 \pmod{p} ,$$
we see that $ \epsilon d^3 + \epsilon d^2 +2$ is nonzero in~$\integer_p$. 
Therefore the voltage is nontrivial, so it generates~$\integer_p$. Hence, \cref{FGL} provides a hamiltonian cycle in $\Cay(G;S)$.

\begin{subcase}
Assume $a = fz$ and $b = fyz^\ell$, with $\ell \neq 0$.
\end{subcase}
Since $\langle a, c \rangle \neq G$ and $\langle b, c \rangle \neq G$, we must have $c \in \langle f, z \rangle w$ and $c \in \langle fy, z \rangle w$.\refnote{3.4GRef}
 So $c \in \langle z \rangle w$; write $c = z^k w$ (with $k \neq 0$ because $S \cap G' = \emptyset$).

We may assume $k, \ell \le (r-1)/2$, by replacing either or both of $b$ and~$c$ with their inverses if necessary. We may also assume $\ell \neq 1$, for otherwise $a \equiv b \pmod{\langle y \rangle}$, so \cref{DoubleEdge} applies. Therefore, we must have $r = 5$\refnote{3.4r=5Ref}
 and $\ell = 2$. We also have $k \in \{1,2\}$.

For $k = 1$, here is a hamiltonian cycle in $\Cay(G/\integer_p; S)$:
\begin{align*}
\begin{array}{ccccccccccc}
\qt{e} \bya {fz} \bybi {yz^4} \byai {fy^2z^3} \byai {yz^2} \byb {fz^4} \\
\byai {z^3} \byai {fz^2} \byai {z} \byai {f} \bybi {yz^3} \\
\bya {fy^2z^4} \bya {y} \bya {fy^2z} \byci {fy^2} \bya {yz} \\
\bya {fy^2z^2} \byb {y^2z^4} \byai {fyz^3} \byai {y^2z^2} \byai {fyz} \\
\byai {y^2} \byai {fyz^4} \byai {y^2z^3} \byai {fyz^2} \byai {y^2z} \\
\byai {fy} \byb {z^2} \bya {fz^3} \bya {z^4} \byc {e}
. \end{array}
\end{align*}
Its voltage is 
\begin{align*}
a b^{-1} a^{-2} b a^{-4} b^{-1} a^3 c^{-1} a^2 b a^{-9} b a^2 c
\end{align*}
Calculating modulo~$y$, the product between the occurrence  of~$c^{-1}$ and the occurrence of~$c$ is 
	$$ a^2 b a^{-9} b a^2
	\equiv (fz)^2 (fz^2) (fz)^{-9}  (fz^2) (fz)^2
	= z^{-1} ,
		 \hbox to 0pt{\color{blue}\hskip 1in(note \ref{3.3.2fzRef})\hss} 
	$$
which does not centralize~$w$. So the occurrence of~$w^{-1}$ in $c^{-1}$ does not cancel the occurrence of~$w$ in~$c$. Therefore the voltage is nontrivial, so it generates~$\integer_p$, so \cref{FGL} applies.

For $k = 2$, here is a hamiltonian cycle in $\Cay(G/\integer_p; S)$: 
\begin{align*}
\begin{array}{ccccccccccc}
\qt{e} \bya {fz} \byb {yz^3} \byb {f} \bya {z} \bya {fz^2} \\
\bya {z^3} \bya {fz^4} \bybi {yz^2} \bya {fy^2z^3} \bya {yz^4} \\
\bya {fy^2} \bya {yz} \bya {fy^2z^2} \byc {fy^2z^4} \bya {y} \\
\bya {fy^2z} \byb {y^2z^3} \bya {fyz^4} \bya {y^2} \bya {fyz} \\
\bya {y^2z^2} \bya {fyz^3} \bya {y^2z^4} \bya {fy} \bya {y^2z} \\
\bya {fyz^2} \byb {z^4} \byai {fz^3} \byai {z^2} \byci {e}
. \end{array}
\end{align*}
Its voltage is 
\begin{align*}
 a b^2 a^4 b^{-1} a^5 c a^2 b a^9 b a^{-2} c^{-1}
.\end{align*}
Calculating modulo~$y$, the product between the occurrence  of~$c$ and the occurrence of~$c^{-1}$ is 
	$$ a^2 b a^9 b a^{-2}
	\equiv (fz)^2 (fz^2) (fz)^9 (fz^2) (fz)^{-2}
	= f z^{13}
	= f z^3 ,
		 \hbox to 0pt{\color{blue}\hskip 1in(note \ref{3.4CommuteRef})\hss} 
	$$
which does not centralize~$w$.\refnote{3.4NotCentRef}
 So the occurrence of~$w^{-1}$ in $c^{-1}$ does not cancel the occurrence of~$w$ in~$c$. Therefore the voltage is nontrivial, so it generates~$\integer_p$, so \cref{FGL} applies.

\begin{case}
Assume $\#S \ge 4$.
\end{case}
Write $S = \{s_1,s_2,\ldots,s_\ell\}$, and let $G_i = \langle s_1,\ldots,s_i \rangle$ for $i = 1,2,\ldots,\ell$. Since $S$ is minimal, we know 
	$$ \{e\} \subsetneq G_1 \subsetneq G_2 \subsetneq \cdots \subsetneq G_\ell \subseteq G . $$ 
Therefore, the number of prime factors of~$|G_i|$ is at least~$i$. Since $|G| = 30p$ is the product of only $4$ primes, and $\ell = \#S \ge 4$, we conclude that $|G_i|$ has exactly~$i$ prime factors, for all~$i$. (In particular, we must have $\#S = 4$.) By permuting the elements of $\{s_1,s_2,\ldots,s_\ell\}$, this implies that if $S_0$ is any subset of~$S$, then $|\langle S_0\rangle|$ is the product of exactly $\#S_0$ primes. In particular, by letting $\#S_0 = 1$, we see that every element of~$S$ must have prime order.
 
Now, choose $\{a,b\} \subset S$ to be a $2$-element generating set of $G/G' \iso \integer_2 \times \integer_r$. 
From the preceding paragraph, we see that we may assume $|a| = 2$ and $|b| = r$ (by interchanging $a$ and~$b$ if necessary). Since $|\langle a,b \rangle|$ is the product of only two primes, we must have $|\langle a,b \rangle| = 2r$, so $\langle a,b \rangle \iso G/G'$. Therefore 
 	$$G = \bigl(\langle a \rangle \times \langle b \rangle \bigr) \ltimes G' .$$

Since $\langle S \rangle = G$, we may choose $s_1 \in S$, such that $s_1 \notin \langle a, b \rangle \, \integer_p$. Then $\langle a,b,s_1 \rangle = \langle a,b \rangle \, \integer_q$.\refnote{4(2qr)Ref}
 Since $a$ centralizes both $a$ and~$b$, but does not centralize $\integer_q$, which is contained in $\langle a, b, s_1 \rangle$, we know that $[a, s_1]$ is nontrivial. Therefore $\langle a, s_1 \rangle$ contains $\langle a,b,s_1 \rangle' = \integer_q$. Then, since $|\langle a, s_1 \rangle|$ is only divisible by two primes, we must have $|\langle a, s_1 \rangle| = 2q$. Also, since $S \cap G' = \emptyset$, we must have $|s_1| \neq q$; therefore $|s_1| = 2$. Hence $2r  \mid|\langle b,s_1\rangle|$, so we must have $| \langle b,s_1\rangle | = 2r$. Therefore
 	$$ [b,s_1]  \in \langle b,s_1\rangle \cap \langle a, b, s_1 \rangle' =   \langle b,s_1\rangle \cap \integer_q = \{e\} ,$$
so $b$ centralizes~$s_1$. It also centralizes~$a$, so $b$ centralizes $\langle a,s_1 \rangle = \integer_2 \ltimes \integer_q$.

Similarly, if we choose $s_2 \in S$ with $s_2 \notin \langle a,b \rangle \, \integer_q$, then $a$ centralizes $\langle b, s_2 \rangle = \integer_r \ltimes \integer_p$. 

Therefore $G = \langle a,s_1 \rangle \times \langle b,s_2 \rangle$, so
	$$\Cay(G;S) \iso \Cay \bigl( \langle a,s_1 \rangle; \{a,s_1\} \bigr) \times \Cay \bigl( \langle b,s_2 \rangle; \{b,s_2\} \bigr) . $$
This is a Cartesian product of hamiltonian graphs and therefore is hamiltonian.
\end{proof}


\newpage

\appendix

\section{Notes to aid the referee}

\bigskip\hrule width\textwidth \bigbreak

\begin{aid} \label{DoubleEdgeAid}
By assumption, there is a hamiltonian cycle $C = (s_i)_{i=1}^n$ in $\Cay(G/N;S)$, such that $s_i = s$, for some~$i$. Replacing $s_i$ with~$t$ does not change the hamiltonian cycle in $\Cay(G/N;S)$, because $t \equiv s = s_i \pmod{N}$, but the voltage of the new cycle is
	$$ s_1s_2 \cdots s_{i-1} t s_{i+1} s_{i+2} \cdots s_n .$$
Since $t \neq s_i$, this is not equal to the voltage of the original cycle. So at least one of the two cycles has a voltage that is $\neq e$. Since $|N|$ is prime, it is generated by any of its nontrivial elements, so \cref{FGL} applies.
\end{aid}

\begin{aid} \label{HamConnInSubgrpRef}
The walk traverses all of the vertices in~$\langle S_0 \rangle$, then the vertices in the coset $a \langle S_0 \rangle$, then the vertices in $a^2 \langle S_0 \rangle$, etc., so it visits all of the vertices in~$G$. Also, note that, for any $h \in H$, we have
	$$ \lower 5pt\hbox{$\Biggl($} \prod_{x \in \langle a \rangle} h^x \lower 5pt\hbox{$\Biggr)$} ^a
	= \prod_{x \in \langle a \rangle} h^{xa}
	= \prod_{x \in \langle a \rangle} h^{x} ,$$
so $\prod_{x \in \langle a \rangle} h^x \in C_H(a)$. Therefore, letting $h = s_1 s_2 \cdots s_m \in H$, we have 
	\begin{align*}
	 (ha)^{|a|} 
	&= a^{|a|} (a^{-|a|} h a^{|a|}) \cdots (a^{-3} h a^3) (a^{-2} h a^2) (a^{-1} h a) 
	\\&= \prod_{x \in \langle a \rangle} h^x
	&& \text{(because $a^{|a|} = e$)}
	\\&\in C_H(a)
	\\&= H \cap Z(G)
	&& \hskip-1.25in \begin{pmatrix}\text{$H \subset \langle S_0 \rangle$ and $\langle S_0 \rangle$ abelian $\Rightarrow$}
	\\ C_H(a) \subset C_H \bigl( \langle S_0, a \rangle \bigr) = C_H(G)
	\end{pmatrix}
	\\&= \{e\}
	, \end{align*}
so the walk is closed. Since the length of the walk is~$|G|$, these facts imply that it is a hamiltonian cycle in $\Cay(G;S)$.
\end{aid} 

\begin{aid} \label{Not3ReflectionsRef}
Suppose $S_0$ is a minimal generating set of~$D_{2pq}$, and $S_0$ contains $3$~reflections $a$, $a t^i$, and~$at^j$, where $t$~is a rotation that generates~$T$. Since $|D_{2pq}|$ is the product of~$3$ primes, and the minimality of~$S_0$ implies 
	$$\langle a \rangle \subsetneq \langle a,  at^i \rangle \subsetneq \langle a,  at^i , at^j \rangle ,$$
we must have $\langle a,  at^i , at^j \rangle = D_{2pq}$. From the minimality of~$S_0$, we know $\langle at^i,  at^j \rangle$ is a proper subgroup~$D_{2pq}$, so we may assume $q \mid (i-j)$ (after interchanging $p$ and~$q$ if necessary). Since $\langle a,  at^i \rangle$ and $\langle a,  at^j \rangle$ must also be proper subgroups (and are not equal to each other), we may assume $p \mid i$ and $q \mid j$ (after interchanging $i$ and~$j$ if necessary). Then
	$$ q \mid  (i-j) + j = i .$$
So $pq \mid i$, which means $a t^i = a$. This contradicts the fact that $a$ and $at^i$ are two different reflections.
\end{aid}

\begin{aid} \label{phi(c)neqTRef}
If $\langle \varphi(c) \rangle = T$, then $\langle c \rangle = T \times \integer_r$ has index~$2$ in~$G$. So $\langle a,c \rangle = G$, which contradicts the fact that $S$ is a minimal generating set.
\end{aid}

\begin{aid} \label{D2qxZrRef}  \ 
\\
\includegraphics{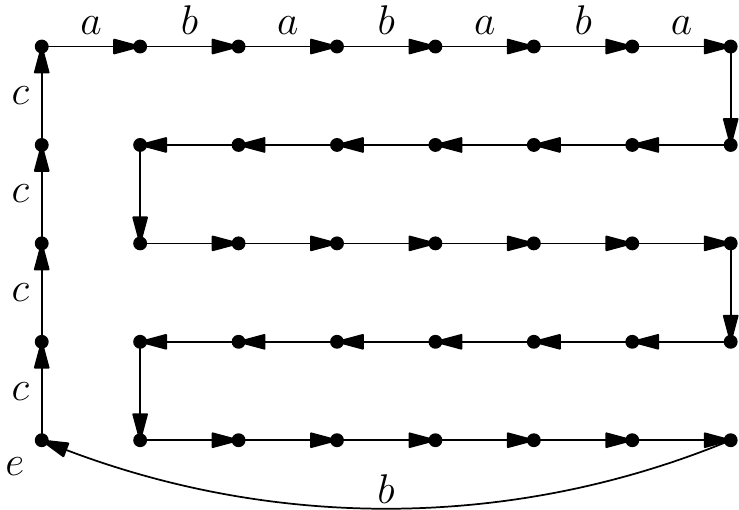}
\end{aid}
%
%
%
%
%
%
%
%
%
%
%

\begin{aid} \label{D2qxZrOtherRef}  
The edges from the new string are dashed.
\\
\includegraphics{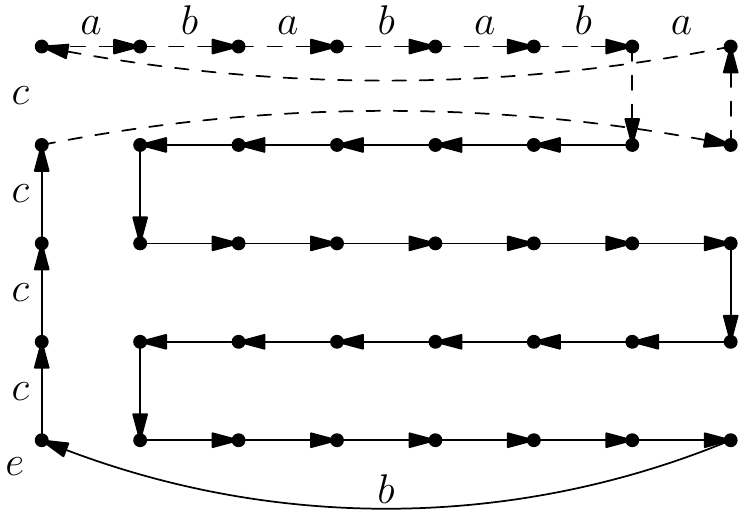}
\end{aid}

\begin{aid} \label{SquareFreeRef}
From the cited theorem of \cite{Hall-ThyGrps} (but replacing the symbol~$r$ with~$\tau$), we know that $G$ is ``metacyclic'', and there exist $a,b \in G$, such that
	\begin{itemize}
	\item $G = \langle b \rangle \ltimes \langle a \rangle$, 
	and
	\item  $\gcd \bigl( (\tau-1)|b|, |a| \bigr) = 1$, where $\tau \in \integer$ is chosen so that $a^b = a^\tau$.
	\end{itemize}

\pref{SquareFree-G'cyclic}~Since $G$ is metacyclic, we know $G'$ is cyclic. In fact, the proof points out that $G' = \langle a \rangle$. (This follows easily from the fact that $\gcd \bigl(\tau-1, |a| \bigr) = 1$.)

\pref{SquareFree-G'NotZ(G)}~Suppose $a^k \in Z(G)$. This means
	$$ e = [a^k , b] = a^{-k} (a^k)^b = a^{-k} a^{k\tau} = a^{(\tau-1)k} , $$
so $|a| \mid (\tau-1)k$. Since $\gcd \bigl(\tau-1, |a| \bigr) = 1$, this implies $|a| \mid k$, so $a^k = e$.

\pref{SquareFree-ZmxG'}~Let $\integer_n = \langle b \rangle$. Then $G = \langle b \rangle \ltimes \langle a \rangle = \integer_n \ltimes G'$.

\pref{SquareFree-tau}~This is one of the conclusions of the cited theorem of \cite{Hall-ThyGrps} (except that we have replaced~$r$ with~$\tau$).
\end{aid}

\begin{aid} \label{15p->dihedralRef}
From \cref{SquareFree}, we may write $G = \langle b \rangle \ltimes \langle a \rangle$ with $|b| = 2$ and $\langle a \rangle = G' \iso \integer_{15p}$. Choose $\tau \in \integer$, such that $a^b = a^\tau$. Since $|b| = 2$, we must have $\tau^2 \equiv 1 \pmod{15p}$, so $\tau \equiv \pm 1$ modulo each prime divisor of~$15p$. Also, we know 
	$$\gcd(\tau-1,15p) = \gcd \bigl( \tau-1, |a| \bigr) = 1 ,$$
which means $\tau \not\equiv 1$ modulo any prime divisor of~$15p$. We conclude that $\tau \equiv -1 \pmod{15p}$, so $G \iso D_{30p}$.
\end{aid}

\begin{aid} \label{G'=15Ref}
From \cref{SquareFree}, we may write $G = \langle b \rangle \ltimes \langle a \rangle$ with $\langle b \rangle \iso \integer_{2p} \iso \integer_2 \times \integer_p$ and $\langle a \rangle = G' \iso \integer_{15}$. Since 
	$$ \gcd \bigl( |\integer_p|, |\Aut(\integer_{15})| \bigr)
	 = \gcd \bigl( p, \phi(15) \bigr)
	  = \gcd(p,8)
	  = 1 ,$$
we know that $\integer_p$ centralizes $\integer_{15}$. So $G = (\integer_2 \ltimes \integer_{15}) \times \integer_p$. Since $G' = \integer_{15}$, the argument of \ref{15p->dihedralRef} implies that $\integer_2 \ltimes \integer_{15} \iso D_{30}$. 
\end{aid}

\begin{aid} \label{G'=pqRef}
From \cref{SquareFree}, we may write $G = \langle b \rangle \ltimes \langle a \rangle$, with $G' = \langle a \rangle$. Choose $\tau \in \integer$, such that $a^b = a^\tau$.

We claim $|a|$ is odd. Suppose not.
From \fullcref{SquareFree}{tau}, we know that $\gcd(\tau-1,|a|) = 1$, so $\tau$ is even. But this contradicts the fact that $\tau$ must be relatively prime to~$|a|$.
 
 So $|G'|$ is an odd divisor of $30p$. In other words, $|G'|$ is a divisor of~$15p$. However, we are assuming that $|G'|$ is not prime, and that it is not~$15$. Therefore, $|G'|$ is either $3p$ or~$5p$.
\end{aid}

\begin{aid} \label{Z2NontrivialRef}
From \cref{SquareFree}, we know $G' \cap Z(G) = \{e\}$, so some element of $\integer_{2r}$ must act nontrivially on $\integer_q$. 
\end{aid}

\begin{aid} \label{ZrCentRef}
We already know that $\integer_r$ centralizes~$\integer_q$. Obviously, it also centralizes $\integer_{2r}$. If it also centralizes~$\integer_p$, then it centralizes all of~$G$, so it is in $Z(G)$. This implies that $G = (\integer_2 \ltimes \integer_{pq}) \times \integer_r$. Since $G' = \integer_{pq}$, the argument of \ref{15p->dihedralRef} implies that $\integer_2 \ltimes \integer_{pq} \iso D_{2pq}$. 
\end{aid}

\begin{aid} \label{gcd(r-1bya)Ref}
Since $r \in \{3,5\}$, we have $r-1 \in \{2,4\}$. Since $15p$ is odd, this implies $\gcd(r-1,15p) = 1$.
\end{aid}

\begin{aid} \label{Z2CentZqRef}
If $q \mid |a|$, then $\langle a \rangle$ contains a subgroup of order~$q$, which is obviously centralized by~$a$. However, $\integer_q$ is the unique subgroup of order~$q$ in~$G$ (since a normal Sylow $q$subgroup is unique). So $a$ centralizes~$\integer_q$. Since the image of~$a$ in $G/G'$ has order~$2$, this implies that $\integer_2$ centralizes~$\integer_q$.
\end{aid}

\begin{aid} \label{aTrivCentRef}
Since $b$ has even order, there is some $k \in \integer$, such that $|b^k| = 2$. Then $\langle a \rangle$ and $\langle b^k \rangle$ are Sylow $2$-subgroups of~$G$, so they must be conjugate. Since $b$ generates $G/G'$ and centralizes $b^k$, this implies there is some $x \in G'$, such that $a^x = b^k$. Writing $G' = C_{G'}(a) \times H$, for some subgroup~$H$, we may write $x = ch$ with $c \in C_{G'}(a)$ and $h \in H$. Then
	$$ a^h =  a^{ch} = a^x = b^k \in \langle b \rangle ,$$
so $a \in \langle b,h \rangle = \langle b \rangle \ltimes H$. Since $\langle a,b \rangle = G$, we conclude that $\langle b \rangle \ltimes H = G$, so $H = G'$. Therefore $C_{G'}(a)$ is trivial.
\end{aid}

\begin{aid} \label{order2optionsRef}
We have either $r = 3$ or $r = 5$. We now show that, for a given choice of~$r$, we need only consider the single situation described in the text.

Since all elements of order~$2$ are conjugate, we may assume $a$ is the unique element of order~$2$ in~$\integer_{2r}$; in other words, $a = x^r$. Since $b$ generates $G/G'$, there is no harm in assuming that the projection of~$b$ to $\integer_{2r}$ is the generator~$x$, so $b = xg'$ for some $g' \in G'$. Since $\langle a,b \rangle = G$, we must have $\langle g' \rangle = G'$, so there is no harm in assuming that $g' = yw$.

We said earlier that $y^x = y^{-1}$. 

Choose $d \in \integer$, such that $w^x = w^d$. Since $a$ does not centralize~$\integer_p$, we know that $x^r$ does not centralize~$\integer_p$, so $d^r \not\equiv 1 \pmod{p}$. Also, we said earlier that $\integer_r$ does not centralize~$\integer_p$, so $x^2$ does not centralize~$\integer_p$, so $d^2 \not\equiv 1 \pmod{p}$. On the other hand, $x^{2r} = e$ does centralize~$\integer_p$, so $d^{2r} \equiv 1 \pmod{p}$. Therefore $d$ is a primitive $(2r)\th$~root of~$1$ in~$\integer_p$. This implies that $d^r \equiv -1 \pmod{p}$. Since $d \not\equiv -1 \pmod{p}$, we may divide by $d + 1$, so, since $r$ is odd, we have
	$$ \sum_{i=0}^{r-1} (-1)^i d^i = \frac{d^r +1}{d+1} \equiv \frac{0}{d+1} \equiv 0 \pmod{p} .$$
\end{aid}

\begin{aid} \label{2rTrivCentRef}
We have $a^{2r} \in G'$ (since $|G/G'| = 2r$), and $a$ obviously centralizes~$a^{2r}$. Since $\langle a \rangle$ has trivial centralizer in~$G'$, this implies $a^{2r} = e$, so $|a| = 2r$. 

Similarly, $|b| = 2r$.
\end{aid}

\begin{aid} \label{a=biRef} \ 
\\
\includegraphics{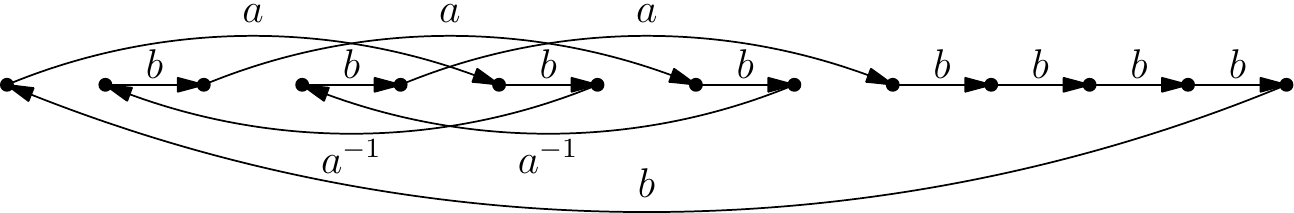}
\end{aid}
%
%
%
%
%
%
%
%

\begin{aid} \label{ror2rthrootRef}
Since $|b| = 2r$, we know $d^{2r} \equiv 1 \pmod{p}$. Also, since $\langle b^2\rangle = \integer_r$ does not centralize~$y$, we have $d^2 \not\equiv 1 \pmod{p}$. Therefore $d$ is either a primitive $r\th$ or $(2r)\th$ root of unity modulo~$p$.
\end{aid}

\begin{aid} \label{ainbiG'-1Ref}
To calculate the exponents of $b$ and~$y$, we can work modulo the normal subgroup~$\langle w \rangle$. Since $\gcd(i,2r) = 1$, we know $1-i$~is odd, so $b^{1-i}$ inverts~$y$ (but $b$ inverts~$y$). Therefore
	\begin{align*}
	 (b^i y)b(y^{-1}b^{-i}) b 
	 &= b^i y^2 b^{2-i}
	 && \text{($b$ inverts $y$)}
	 \\&= b^2 y^{-2}
	 && \begin{pmatrix}
	 \text{$\gcd(i,2r) = 1$, so $2-i$~is odd,}
	\\ \text{so $b^{2-i}$ inverts~$y$}
	\end{pmatrix}
	. \end{align*}

Now, to calculate the exponent of~$y$, we can work modulo the normal subgroup~$\langle y \rangle$. Since $w^b = w^d$, we have
	$$ (b^i w)b(w^{-1} b^{-i}) b
	 = b^{i+1} w^{d-1} b^{1-i}
	 = b^{2} w^{(d-1)d^{1-i}}
	 . $$
\end{aid}

\begin{aid} \label{ainbiG'-2Ref}
To calculate the exponents of $b$ and~$y$, we work modulo~$\langle w \rangle$. Since $b$ inverts~$y$, we know $b^2$ centralizes~$y$, so 
	$$ (b^2 y^{-2})^{(i-1)/2}
	= (b^2)^{(i-1)/2} (y^{-2})^{(i-1)/2}
	= b^{i-1} y^{-(i-1)}
	. $$

\newcommand{\ww}{\underline{w}}
\newcommand{\bb}{\underline{b}}
Now, to calculate the exponent of~$w$, we  can work modulo the normal subgroup~$\langle y \rangle$. For convenience, let $\bb = b^2$, $\ww = w^{(d-1)d^{1-i}}$, and $i' = (i-1)/2$. Then
	\begin{align*}
	 (b^2 w^{(d-1)d^{1-i}})^{(i-1)/2}
	&= (\bb \ww)^{i'}
	\\&= \bb^{i'} (\bb^{-(i'-1)} \ww \bb^{i'-1} ) (\bb^{-(i'-2)} \ww \bb^{i'-2} ) \cdots 
	(\bb^{-1} \ww \bb^1 )
	(\bb^{-0} \ww \bb^0 )
	\\&= b^{i-1} (b^{-(i-3)} \ww b^{i-3} )  (b^{-(i-5)} \ww b^{i-5} ) \cdots  (b^{-2} \ww b^2 )  (b^{-0} \ww b^{0} )
	\\&= b^{i-1} (\ww^{d^{i-3}} ) (\ww^{d^{i-5}} ) \cdots (\ww^{d^{2}} ) (\ww^{d^{0}} )
	\\&= b^{i-1} \ww^{d^{i-3} + d^{i-5} + \cdots + d^{2} + 1}
	\\&= b^{i-1} w^{(d-1)d^{1-i}(d^{i-3} + d^{i-5} + \cdots + d^{2} + 1)}
	. \end{align*}
\end{aid}

\newcommand{\ww}{\underline{w}}
\begin{aid}  \label{ainbiG'-3Ref}
For convenience, let $\ww = w^{(d-1)(d^{i-3} + d^{i-5} + \cdots + d^2 + 1)}$. Then
	\begin{align*}
	\bigl( b^{i-1} y^{-(i-1)}  & w^{(d-1)d^{1-i}(d^{i-3} + d^{i-5} + \cdots + d^2 + 1)}\bigr) (b^i yw)
	\\&= \bigl( b^{i-1} y^{-(i-1)}  \ww^{d^{1-i}}\bigr) (b^i yw)
	\\&= \bigl( b^{2i-1} y^{i-1}  (\ww^{d^{1-i}})^{d^i}\bigr) (yw)
	&&	\text{($b^i$ inverts~$y$, since $i$ is odd)}
	\\&= \ b^{2i-1} y^{(i-1)+1}  \ww^{d} (w)
	&&	\begin{pmatrix}
	\text{$y$ commutes with~$w$,}
	\\ \text{since both are in $\integer_{pq}$}
	\end{pmatrix}
	. \end{align*}
Also, we have 
	$$ \ww^{d} (w)
	= (w^{(d-1)(d^{i-3} + d^{i-5} + \cdots + d^2 + 1)})^{d}(w)
	= w^{(d-1)d (d^{i-3} + d^{i-5} + \cdots + d^2 + 1)+ 1}
	. $$
\end{aid}

\begin{aid} \label{q=3&r=5&i=3Ref}
Recall that $\{q,r\} = \{3,5\}$. Since $q \mid i$ and $i < r$, we must have $q < r$, so $q = 3$ and $r = 5$. Then, since $q \mid i$ and $i < r$, we have $3 \mid i$ and $i < 5$, so it is obvious that $i = 3$.
\end{aid}

\begin{aid} \label{S=3-abcRef}
Let $c$ be an element of~$S$ with nontrivial projection to~$\integer_r$, so $\integer_r \subset \langle c \rangle$. Since $S$ is minimal and $\# \bigl( S \smallsetminus \{c\} \bigr) > 1$, we know that $|\qt{G}/\langle \qt{c} \rangle|$ cannot be prime. Therefore $\langle \qt{c} \rangle = \integer_r$.

The other elements of~$S$ must have trivial projection to~$\integer_r$. (Otherwise, the previous paragraph implies they belong to $\integer_r = \langle \qt{c} \rangle$, contradicting the minimality of~$\qt{S}$. So $\qt{a}, \qt{b} \in D_{2q}$.
\end{aid}

\begin{aid} \label{S=3-c=rRef}
We have $c^r \in \integer_p$ (since $\qt{c}^r = \qt{e}$), and $c$ obviously centralizes~$c^r$. Since $\langle \qt{c} \rangle = \integer_r$ acts nontrivially on~$\integer_p$, and hence has trivial centralizer in~$\integer_p$, this implies $c^r= e$, so $|c| = r$. 

This implies that $\langle c \rangle$ is a Sylow $r$-subgroup of~$G$, so it is conjugate to any other Sylow $r$-subgroup, including~$\integer_r$.
\end{aid}

\begin{aid} \label{DihedralxZrRef}
If $\integer_r \subset Z(G)$, then $G = \langle a,b \rangle \times \integer_r$. 
Also, since $|a| = |b| = 2$, we know that $\langle a,b \rangle$ is a dihedral group.
Therefore \cref{DihedralxZr} applies.
\end{aid}

\begin{aid} \label{WcyclesRef} 
Edges not in~$W$ are dashed.
\\[5pt]
\includegraphics{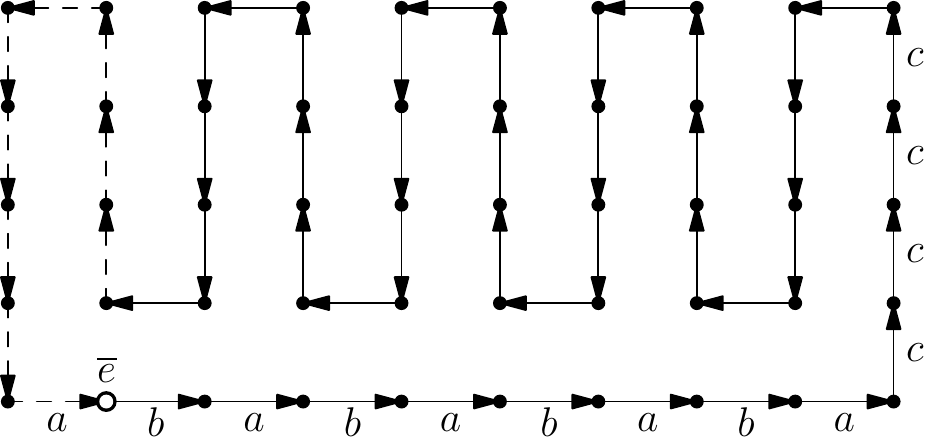}
\\[10pt]
\includegraphics{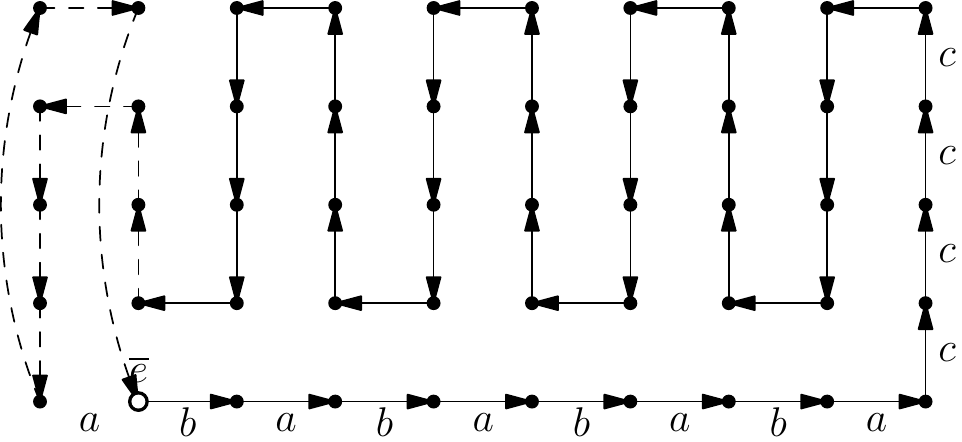}
\end{aid}

\begin{aid} \label{<ab>=DxZRef}
Let $H = \langle a, b \rangle$. Since $\langle \qt{a}, \qt{b} \rangle = \qt{G}$, we know $2qr \mid |H|$. On the other hand, the minimality of~$S$ implies $H \neq G$, so $H$ is a proper divisor of $|G| = 2pqr$. Therefore $|H| = 2qr$. Since $G$ is solvable, any two Hall subgroups of the same order are conjugate \cite[Thm.~9.3.1(2), p.~141]{Hall-ThyGrps}, so $H$ is conjugate to $D_{2q} \times \integer_r$.
\end{aid}

\begin{aid} \label{fyRef}
Let $\varphi \colon \langle a,b \rangle \to D_{2q}$ be the projection with kernel~$\integer_r$. 

\setcounter{case}{0}

\begin{case} \label{fyRef-aProjTriv}
Assume the projection of~$a$ to~$\integer_r$ is trivial. 
\end{case}
This means $a = f$. Then $b$ must project nontrivially to~$\integer_r$ (since $\langle a,b \rangle = D_{2q} \times \integer_r$). Therefore, we may assume the projection of~$b$ to~$\integer_r$ is~$z$ (since every nontrivial element of~$\integer_r$ is a generator). Therefore $b$ is either $yz$ or $fyz$, depending on whether $\varphi(b)$ is $y$ or~$fy$, respectively.

\begin{case}
Assume the projection of~$a$ to~$\integer_r$ is nontrivial. 
\end{case}
We may assume $a = fz$ (since every nontrivial element of~$\integer_r$ is a generator). 

We have $b = \varphi(b) \, z^\ell$ for some $\ell \in \integer$, and we wish to show that we may assume $\ell \not\equiv 0 \pmod{r}$. That is, we wish to show that we may assume $b \neq \varphi(b)$.
	\begin{itemize}
	\item Since $y \notin S$, we know that $b \neq \varphi(b)$ if $\varphi(b) = y$.
	\item If $b = \varphi(b) = fy$, then interchanging $a$ and~$b$ would put us in \cref{fyRef-aProjTriv}.
	\end{itemize}
\end{aid}

\begin{aid} \label{a=f&b=yz-i=0Ref} 
Suppose $i \neq 0$, which means $i = 1$. Since $y$ and~$z$ commute, we have $\langle y z \rangle = \langle y \rangle \times \langle  z \rangle$. Therefore
	$$ \langle b, c \rangle = \langle y, z, f y^j z^k w \rangle = \langle y, z, f w \rangle . $$
This contains 
	$$ ( f w)^{-1} (f w)^z = ( f w)^{-1} (f w^{d}) = w^{d-1} .$$
Since $d \neq 1$, we have $\langle w^{d-1} \rangle = \integer_p$, so $\langle b,c \rangle$ contains~$w$. Since it also contains $y$, $z$, and~$fw$, we conclude that $\langle b,c \rangle = G$.
\end{aid}

\begin{aid} \label{a=f&b=yz-k=1aRef}
We have
	\begin{align*}
	\bigl( (f)(yz)^{-(r-1)} (f) \bigr) (yz)^{r-1}
	&= f^2 (y^{-1}z) ^{-(r-1)}  (yz)^{r-1}
	&& \text{($f$ inverts~$y$ and centralizes~$z$)}
	\\&= y^{2(r-1)}
	&& \text{($|f| = 2$ and $y$ commutes with~$z$)}
	.\end{align*}
Also, $(yz)^{-1} (y^j zw) = y^{j-1} w$, since $y$ commutes with~$z$.
\end{aid}

\begin{aid}  \label{a=f&b=yz-k=1bRef}
Since $|y| = q$, it suffices to check (for each of the two possible values of~$q$) that the given exponent of~$y$ is congruent to $j-2$, modulo~$q$:
	\begin{itemize}
	\item If $q = 5$, then $j+3 \equiv j -2 \pmod{q}$.
	\item If $q = 3$, then $j+7 \equiv j -2 \pmod{q}$.
	\end{itemize}
\end{aid}

\begin{aid} \label{a=f&b=yz-k=1jaRef}
We have
	\begin{align*}
	\bigl( (f)(yz)^{-(r-1)} (f) \bigr) (yz)^{r-1}
	&= f^2 (y^{-1}z) ^{-(r-1)}  (yz)^{r-1}
	&& \text{($f$ inverts~$y$ and centralizes~$z$)}
	\\&= y^{2(r-1)}
	&& \text{($|f| = 2$ and $y$ commutes with~$z$)}
	.\end{align*}
Also, 
	\begin{align*}
	(y^2 zw)^2
	 &=  (y^2 zw)(y^2 zw)
	 \\&=(y^4 zw)( zw)
	 && \text{($y$ commutes with both~$z$ and~$w$)}
	 \\&=y^4 z^2w^{d + 1}
	 && \text{($w^z = w^{d}$)}
	, \end{align*}
so
	$$ (yz)^{-2}(y^2 zw)^2 
	= (yz)^{-2}(y^4 z^2w^{d + 1})
	= y^2w^{d + 1} ,$$
since $y$ commutes with~$z$.
\end{aid}

\begin{aid}  \label{a=f&b=yz-k=1jbRef}
Since $|y| = q$, it suffices to check (for each of the two possible values of~$q$) that the given exponent of~$y$ is congruent to $1$, modulo~$q$:
	\begin{itemize}
	\item If $q = 5$, then $6 \equiv 1 \pmod{q}$.
	\item If $q = 3$, then $10 \equiv 1 \pmod{q}$.
	\end{itemize}
\end{aid}

\begin{aid}  \label{a=f&b=yz-k=1jcRef}
Since $d$ is a primitive $r\th$ root of unity in~$\integer_p$, we know $d \not\equiv -1 \pmod{p}$. Therefore $w^{d+1}$ is nontrivial, and hence generates~$\integer_p$.
\end{aid}

\begin{aid} \label{a=f&b=yz-k>1Ref}
Since $y$ commutes with $z$, we have
	\begin{align*}
	 (fz)^4 &= f^4 z^4 = z^4 , \\
	 fz^{-1} fz &= f^2 = e ,\\
	 w^{-1} z^{-2} f z^2 w &= w^{-1} f w = w^{-1 + \epsilon } f , \\
	 z^{-1} f z &= f , \\
	 (fzfz^{-1})^2 &= (f^2)^2 = e^2 = e
	 . \end{align*}
Also,
	\begin{align*}
	(fw^{-1} z^{-2})^2
	&= (fw^{-1} z^{-2})(fw^{-1} z^{-2})
	\\&= fw^{-1}f w^{-d^2} z^{-4}
	&& \text{($z$ commutes with~$f$, but $w^z = w^{d}$)}
	\\&= f^2 w^{-\epsilon -d^2} z^{-4}
	&& \text{($w^f = w^\epsilon $)}
	\\&= w^{-(\epsilon +d^2)} z^{-4}
	&& \text{($|f| = 2$)}
	. \end{align*}
\end{aid}

\begin{aid} \label{jnot0Ref}
Since $y$ centralizes both $z$ and~$w$ (and $j \neq 0$), we have
	$$ \langle c \rangle = \langle y^j z^2 w \rangle = \langle y \rangle \times \langle z^2 w \rangle .$$
Therefore $ \langle a, c \rangle = \langle f, y, z^2 w \rangle$. 

Since $f$ centralizes~$z$, this contains 
	$$ (z^2 w)^{-1} (z^2 w)^f = (z^2 w)^{-1} (z^2 w^f) = [w,f] .$$
If $f$ does not centralize~$\integer_p$, then $[w,f]$ is nontrivial, so it generates $\integer_p = \langle w \rangle$. This implies that $\langle a, c \rangle$ contains~$w$. Since it also contains $a$, $c$, and~$z^2 w$, this would imply that $\langle a, c \rangle = G$, which is a contradiction. Therefore $f$ centralizes~$\integer_p$. 

So $f$ and~$y$ each centralize both~$z$ and~$w$. Therefore 
	$$ G = \langle f,y \rangle \times \langle z, w \rangle = D_{2q} \times (\integer_r \ltimes \integer_p) = D_{6} \times (\integer_5 \ltimes \integer_p) .$$ 
\end{aid}

\begin{aid} \label{cInfyzRef}
Since $z$ commutes with $f$ and~$y$, we have 
	$\langle fyz \rangle = \langle fy \rangle \times \langle z \rangle$. 
Also, since $c = f^i y^j z^k w$, we have $c \in \langle fy, z \rangle y^\ell w$ for some $\ell \in \integer$. Therefore
	$$ \langle b, c \rangle = \langle fy, z, c \rangle = \langle fy, z, y^\ell w \rangle .$$
This contains 
	\begin{align*}
	(y^\ell w)^{-1} (y^\ell w)^z 
	&= (y^\ell w)^{-1} (y^\ell w^z) 
	&& \text{($z$ centralizes~$y$)}
	\\& = w^{-1} w^z
	\\&= [w,z]
	. \end{align*}
Since $\integer_r$ does not centralize~$\integer_p$, this commutator is nontrivial, so it generates $\integer_p = \langle w \rangle$. Therefore $\langle b, c \rangle$ contains~$w$. It also contains $fy$, $z$, and $y^\ell w$. If $\ell \neq 0$, this implies $\langle b,c \rangle = G$, which contradicts the minimality of~$S$.

Therefore, we must have $\ell = 0$, so 
	$ c \in \langle fy, z \rangle y^\ell w = \langle fy, z \rangle w$.
\end{aid}

\begin{aid} \label{3.2.1b^r-1Ref} \ 
	\begin{itemize}
	\item $z$ commutes with both $f$ and~$y$, so $(fyz)^{r-1} = (fy)^{r-1}z^{r-1}$
	\item $fy$ is a reflection, so it has order~$2$, so $(fy)^{r-1} = e$, since $r-1$ is even.
	\item $z^r = e$, since $z \in \integer_r$, so $z^{r-1} = z^{-1}$.
	\end{itemize}
\end{aid}

\begin{aid} \label{3.3.2HamRef}
Modulo $G' = \langle y, w \rangle$, we have $a \equiv f$, $b \equiv fz$, and $c \equiv z$.
Since $f$ commutes with~$z$, we have
	$$ (ac)^{r-1} ab) \equiv (fz)^{r-1} f \, fz = f^{r+1} z^r = e , $$
since $|f| = 2$, $r+1$ is even, and $|z| = r$. Therefore, the walk in $\Cay(G/G'; S)$ is closed.
\\[10pt]
\includegraphics{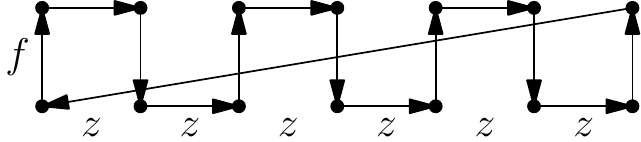}
\end{aid}
%
%
%
%
%
%
%
%

\begin{aid} \label{3.2.2aRef} \ \\
$(ac)^{r-1} \, a \, b
= (ac)^{r-1} \bigl( (ac) (ac)^{-1} \bigr) \, a \, b
= \bigl( (ac)^{r-1} (ac) \bigr) (c^{-1}a^{-1}) \, a \, b
	= (ac)^{r} \, (c^{-1} \, b)$
\end{aid}

\begin{aid} \label{3.2.2bRef} 
	\begin{align*}
	(fzw)^{r}
	&= \bigl( (fz) w \bigr) \bigl( (fz) w \bigr) \cdots \bigl( (fz) w \bigr) \bigl( (fz) w \bigr)
	\\&=(fz)^r  \bigl( (fz)^{-(r-1)} w (fz)^{r-1} \bigr)  \bigl( (fz)^{-(r-2)} w (fz)^{r-2} \bigr) \cdots
		 \bigl( (fz)^{-1} w (fz)^{1} \bigr)  \bigl( (fz)^{-0} w (fz)^{0} \bigr) 
	\\&= f^r z^r w^{(\epsilon d)^{r-1} +(\epsilon d)^{r-2} + \cdots + 1}
	. \end{align*}
\end{aid}

\begin{aid} \label{3.2.2cRef} \ 
	\begin{itemize}
	\item $f^r = f$ because $|f| = 2$ and $r$ is odd.
	\item $|z| = r$ and $z$ commutes with both~$f$ and~$y$.
	\end{itemize}
\end{aid}

\begin{aid} \label{3.2.2dRef} 
Let $\omega \in \integer$. If
	$$ \omega^{r-2} + \omega^{r-3} + \cdots \omega + 1 \equiv 0 \pmod{p} ,$$
then 
	$$\omega^{r-1}-1 = (\omega - 1)(\omega^{r-2} + \omega^{r-3} + \cdots \omega + 1)
	 \equiv (\omega - 1)(0) = 0 \pmod{p} ,$$
so $\omega$ is an $(r-1)^{\text{st}}$ root of unity in~$\integer_p$. Therefore, it cannot be a primitive $r\th$ or $(2r)\th$ root of unity.
\end{aid}

\begin{aid} \label{3.2.2k>1Ref} 
We have
	\begin{align*}
	(z^2 w)^{-1} f (z^2 w)
	&= (w^{-1} z^{-2}) f (z^2 w)
	\\&= w^{-1} f w
	&& \text{($z$ commutes with~$f$)}
	\\&= w^{-1} (f w f) f
	&& (f^2 = e)
	\\&= w^{\epsilon - 1} f
	, 
	\\[5pt] (fz)^{-1}f (fz) 
	&=  (z^{-1} f^{-1}) f (fz) 
	\\&=  f 
	&& \text{($f$ and $z$ commute)}
	. 
	\intertext{and}
	\bigl(f(z^2 w)^{-1} \bigr)^2
	&= (f w^{-1} z^{-2}) (f w^{-1} z^{-2}) 	 
	\\&= (f w^{-1}f) ( z^{-2} w^{-1} z^2)  z^{-4}
	&& \text{($f$ and~$z$ commute)}
	\\&= (w^{-\epsilon }) (w^{-d^2} )  z^{-4}
	\\&= w^{-(\epsilon + d^2)} z^{-4}
	. \end{align*}
\end{aid}

\begin{aid} \label{3.2.3CentRef}
Since $0 \le i < 2$ and we are assuming that $i \neq 0$, we have $c = fyz^k w$, so
	$$ \langle a, c \rangle = \langle f, fy z^k w \rangle = \langle f, y z^k w \rangle .$$
Since $y$ commutes with both $z$ and~$w$, we have
	$$ \langle y z^k w \rangle = \langle y \rangle \times \langle z^k w \rangle ,$$
so $\langle a,c \rangle$ contains both $y$ and~$z^k w$. Therefore, since $f$ centralizes~$z$, it also contains
	$$ (z^k w)^{-1} (z^k w)^f = (w^{-1} z^{-k}) (z^k w^f) = w^{-1} w^f = [w,f] .$$
If $f$ does not centralize~$w$, then this commutator is nontrivial, so it generates $\integer_p = \langle w \rangle$. This implies that $\langle a,c \rangle$ contains~$w$. Since it also contains $f$, $y$, and $z^k w$ (with $k \neq 0$), we conclude that $\langle a, c \rangle = G$. This is a contradiction. So $f$ must centralize~$w$.

Hence, $f$ and $y$ each centralize both~$z$ and~$w$, so 
	$$ G = \langle f,y \rangle \times \langle z, w \rangle = D_{2q} \times (\integer_r \ltimes \integer_p) .$$
\end{aid}

\begin{aid} \label{3.2.3VoltRef}
	\begin{align*}
	\bigl( z^4 w z w \bigr)^3 
	&= \bigl( (z^{-1} w z) w \bigr)^3 
	&& (|z| = r = 5)
	\\&= \bigl( w^d w \bigr)^3 
	\\&= w^{3(d+1)}
	. \end{align*}
\end{aid}

\begin{aid} \label{3.2.3GenRef}
$d$ is a primitive $r\th$ root of unity in~$\integer_p$, so $d+1 \not\equiv 0 \pmod{p}$. Since $p \ge 7$, this implies $3(d+1) \not\equiv 0 \pmod{p}$. Therefore $w^{3(d+1)}$ is nontrivial, and hence generates~$\integer_p$.
\end{aid}

\begin{aid} \label{3.3GRef}
We have $c = f^i y^j z^k w$. 
\medskip

We claim that $j = 0$ (which means $c \in \langle f,z \rangle w$).
Since $z$ commutes with $f$, we have 
	$$\langle a \rangle = \langle fz \rangle = \langle f \rangle \times \langle z \rangle .$$
Therefore 
	$$ \langle a,c \rangle = \langle f, z, f^i y^j z^k w \rangle = \langle f, z, y^j w \rangle ,$$
which contains
	$$ (y^j w)^{-1} (y^j w)^z = (w^{-1} y^{-j})(y^j w^z) = w^{-1} w^z = [w, z] .$$
Since $\integer_r$ does not centralize~$\integer_p$, this commutator is nontrivial, so it generates $\integer_p = \langle w \rangle$. Therefore $\langle a,c \rangle$ contains~$w$. So it contains $(y^jw)w^{-1} = y^j$. 

If $j \neq 0$, this implies that $\langle a,c \rangle$ contains~$y$. Since it also contains $f$, $z$, and~$w$, we would have $\langle a, c \rangle = G$, which is a contradiction. Therefore $j = 0$, as claimed.

\medskip

We claim that $i = 0$ (which means $c \in \langle y,z \rangle w$).
Since $z$ commutes with $y$ (and $\ell \neq 0$), we have 
	$$ \langle b \rangle = \langle y z^\ell \rangle = \langle y \rangle \times  \langle z^\ell \rangle = \langle y \rangle \times  \langle z \rangle .$$
Therefore 
	$$ \langle b,c \rangle = \langle y, z, f^i y^j z^k w \rangle = \langle y, z, f^i w \rangle ,$$
which contains
	$$ (f^i w)^{-1} (f^i w)^z = (w^{-1} f^{-i})(f^i w^z) = w^{-1} w^z = [w, z] .$$
Since $\integer_r$ does not centralize~$\integer_p$, this commutator is nontrivial, so it generates $\integer_p = \langle w \rangle$. Therefore $\langle b,c \rangle$ contains~$w$. So it contains $(f^iw)w^{-1} = f^i$. 

If $i \neq 0$, this implies that $\langle b,c \rangle$ contains~$f$. Since it also contains $y$, $z$, and~$w$, we would have $\langle b, c \rangle = G$, which is a contradiction. Therefore $i = 0$, as claimed.

\medskip

Since $i = 0$ and $j = 0$, we have $c = z^k w$.
\end{aid}

\begin{aid} \label{3.3.2r=5Ref}
If $r = 3$, then $(r-1)/2 = 1$, so $\ell = k = 1$, contradicting the fact that $\ell \neq k$.

Thus, we must have $r = 5$, so $(r-1)/2 = 2$. Since $\ell \neq k$, we must have $\{\ell,k\} = \{1,2\}$.
\end{aid}

\begin{aid} \label{3.3.2fzRef}
Recall that $f$ commutes with~$z$, and $f^2 = e$
\end{aid}

\begin{aid} \label{3.3.2zRef}
	\begin{align*}
	\bigl( z^{-1} w fz^{-2}w \bigr)^2
	&= \bigl( (z^{-1} w z) f (z^{-3}w z^3) z^{-3} \bigr)^2
	&& \text{($f$ commutes with~$z$)}
	\\&= \bigl( (w^{d}) f (w^{d^3}) z^{-3} \bigr)^2
	\\&= \bigl( f w^{d^3+ \epsilon d} z^{-3} \bigr)^2
	\\&= \bigl( f w^{d^3+ \epsilon d} z^{-3} \bigr)\bigl( f w^{d^3+ \epsilon d} z^{-3} \bigr)
	\\&= ( f w^{d^3+ \epsilon d} f) (z^{-3} w^{d^3+ \epsilon d} z^3 ) z^{-6}
	&& \text{($f$ commutes with~$z$)}
	\\&= ( w^{\epsilon (d^3+ \epsilon d)}) (w^{d^3(d^3+ \epsilon d)} ) z^{-6}
	\\&= ( w^{d^6+ \epsilon d^4 + \epsilon d^3+ d} ) z^{-6}
	&& (\epsilon ^2 = 1)
	. \end{align*}
\end{aid}

\begin{aid} \label{3.3.2z4Ref}
Since $d$ is an $r\th$ root of unity in~$\integer_p$, and $r = 5$, we have 
	$ d^6 \equiv d \pmod{p} $, so, modulo~$p$, we have
	$$ d^6+ \epsilon d^4 + \epsilon d^3+ d 
	\equiv d + \epsilon d^4 + \epsilon d^3+ d 
	= \epsilon d^4 + \epsilon d^3+ 2d 
	= d( \epsilon d^3 + \epsilon d^2+ 2) 
	. $$ 

Also, since $|z| = r = 5$, we have $z^{-6} = z^4$.
\end{aid}

\begin{aid} \label{3.4GRef}
If we write $c = f^i y^j z^k w$, then, exactly as in note~\ref{3.3GRef}, we must have $j = 0$ (which means $c \in \langle f,z \rangle w$).

\medskip

We may also write write $c = (fy)^i y^{j'} z^k w$. 
We claim that $j' = 0$ (which means $c \in \langle fy,z \rangle w$).
Since $z$ commutes with both $f$ and~$y$ (and $\ell \neq 0$), we have 
	$$ \langle b \rangle = \langle fy z^\ell \rangle = \langle fy \rangle \times  \langle z^\ell \rangle = \langle fy \rangle \times  \langle z \rangle .$$
Therefore 
	$$ \langle b,c \rangle = \langle fy, z, (fy)^i y^{j'} z^k w \rangle = \langle fy, z, y^{j'}w \rangle ,$$
which contains
	$$ (y^{j'} w)^{-1} (y^{j'} w)^z = (w^{-1} y^{-j'})(y^{j'} w^z) = w^{-1} w^z = [w, z] .$$
Since $\integer_r$ does not centralize~$\integer_p$, this commutator is nontrivial, so it generates $\integer_p = \langle w \rangle$. Therefore $\langle b,c \rangle$ contains~$w$. So it contains $(y^{j'}w)w^{-1} = y^{j'}$. 

If $j' \neq 0$, this implies that $\langle b,c \rangle$ contains~$y$. Since it also contains $fy$, $z$, and~$w$, we would have $\langle b, c \rangle = G$, which is a contradiction. Therefore $j' = 0$, as claimed.

\medskip

Therefore 
	$$c \in \langle f,z \rangle w \cap \langle fy,z \rangle w 
	= \bigl( \langle f,z \rangle  \cap \langle fy,z \rangle \bigr) w 
	= \langle z \rangle w 
	. $$
\end{aid}

\begin{aid} \label{3.4r=5Ref}
If $r = 3$, we have $1 < \ell \le (r-1)/2 = 1$, which is impossible. Therefore $r = 5$.
So we have $1 < \ell \le (r-1)/2 = 2$, which implies $\ell = 2$. Also, since $1\le k \le  (r-1)/2 = 2$, we have $k \in \{1,2\}$.
\end{aid}

\begin{aid} \label{3.4CommuteRef}
Recall that $f$ commutes with~$z$, and $f^2 = e$.  Also, we have $z^5 = z^r = e$, so $z^{13} = z^3$.
\end{aid}

\begin{aid} \label{3.4NotCentRef}
We have
	$$ (fz^3)^{-1} w(fz^3) 
	= z^{-3} (f^{-1} w f) z^3
	= z^{-3} w^{\epsilon } z^3
	= w^{\epsilon d^3}
	. $$
Since $d$ is a primitive $r\th$ root of unity in~$\integer_p$, we know $d^3 \not\equiv \pm1 \pmod{p}$. Therefore $\epsilon d^3 \not\equiv 1 \pmod{p}$, so $ (fz^3)^{-1} w(fz^3)  \neq w$.
\end{aid}

%
%

\begin{aid} \label{4(2qr)Ref}
Since $|\langle a,b,s_1 \rangle|$ is the product of only three primes (and is divisible by $|\langle a,b \rangle| = 2r$), it must be either $2qr$ or $2pr$.

However, if $|\langle a,b,s_1 \rangle| = 2pr$, then $\langle a,b,s_1 \rangle$ contains $\integer_p$ (since $\integer_p$ is a normal Sylow $p$-subgroup of~$G$, and hence is the unique subgroup of order~$p$ in~$G$). So 
	$$ \langle a,b,s_1 \rangle \supset \langle a,b \rangle \, \integer_p .$$
Since they have the same order, these two subgroups must be equal, so 
	$$ s_1 \in  \langle a,b,s_1 \rangle = \langle a,b \rangle \, \integer_p .$$
This contradicts the choice of~$s_1$.

Therefore $|\langle a,b,s_1 \rangle| = 2qr$. Since $\integer_q$ is a normal Sylow $q$-subgroup of~$G$, we know that it is the unique subgroup of order~$q$ in~$G$. So $\integer_q \subset \langle a,b,s_1 \rangle$. Hence (by comparing orders) we must have $\langle a,b,s_1 \rangle = \langle a,b \rangle \, \integer_q$.
\end{aid}

\end{document}